\documentclass[12pt,reqno]{amsart}
 
\usepackage[left=0.8in, right=0.8in, top=1in,bottom=1in]{geometry}
\usepackage[dvipsnames]{xcolor}
\usepackage[mathscr]{eucal}
\usepackage{latexsym,bm,amsfonts,amssymb,pifont,  bm,  stmaryrd}  

\usepackage{graphicx}   
\usepackage{color}
\usepackage{hyperref}   
\hypersetup{   
     colorlinks   = true, 
     citecolor    = blue, 
     linkcolor    = blue
}
\usepackage{bbm}  
  
\newcommand{\red}[1]{{\color{red} #1}}
\newcommand{\blue}[1]{{\color{blue} #1}}

\usepackage{pdfsync}
\usepackage[font={scriptsize}]{caption}

\usepackage{enumitem}
\usepackage{pdfsync, tcolorbox}

\usepackage{tikz}
\makeatletter

\newtheorem{theorem}{Theorem}[section]
\newtheorem{Def}[theorem]{Definition}
\newtheorem{notation}[theorem]{Notation}
\newtheorem{thm}[theorem]{Theorem}
\newtheorem{prop}[theorem]{Proposition}

\newtheorem{lemma}[theorem]{Lemma}

\theoremstyle{remark}
\newtheorem{remark}[theorem]{Remark}

\numberwithin{equation}{section}
\renewcommand{\theequation}
{\arabic{section}.\arabic{equation}}

\newcommand{\hb}[1]{\textcolor{blue}{#1}}

\def\RR{\mathbb{R}}

\def\EE{\mathbb{E}}

\def\bfu{{\bf u}}
\def\bfv{{\bf v}}

\def\bfz{{\bf z}}

\newcommand{\ca}{{\mathcal A}}

\newcommand{\ch}{{\mathcal H}}

\newcommand{\cp}{{\mathcal P}}

\newcommand{\cv}{{\mathcal V}}

\def\al{{\alpha}}

\def\ga{{\gamma}}

\newcommand{\lp}{\left(}
\newcommand{\rp}{\right)}
\newcommand{\lc}{\left[}
\newcommand{\rc}{\right]}

\makeatletter
\def\l@subsection{\@tocline{2}{0pt}{2.5pc}{5pc}{}}
\makeatother

\begin{document}
\title[Fractional Brownian Volterra signatures]
{Volterra equations driven by rough signals 3:\\
Probabilistic construction of the Volterra rough path for fractional Brownian motions} 
\date{\today}   

\author[F. Harang \and S. Tindel \and X. Wang]
{Fabian Harang \and Samy Tindel \and  Xiaohua Wang}
 \address{Samy Tindel, Xiaohua Wang: 
 Department of Mathematics,
Purdue University,
150 N. University Street,
W. Lafayette, IN 47907,
USA.}
\email{stindel@purdue.edu, wang3296@purdue.edu}

\address{Fabian A. Harang: Department of Economics, BI Norwegian Business School, Handelshøyskolen BI, 0442, Oslo, Norway.}
\email{fabian.a.harang@bi.no}

\thanks{S. Tindel is supported by the NSF grant  DMS-1952966.}


\begin{abstract}
Based on the recent development of the framework of Volterra rough paths \cite{HarangTindel}, we consider here the probabilistic construction of the Volterra rough path associated to the fractional Brownian motion with $H>\frac{1}{2}$ and for the standard Brownian motion. The Volterra kernel $k(t,s)$ is allowed to be singular, and behaving similar to $|t-s|^{-\gamma}$ for some $\gamma\geq 0$. The construction is done in both the Stratonovich and It\^o sense. It  is based on a modified Garsia-Rodemich-Romsey lemma which has an interest in its own right, as well as tools from Malliavin calculus. A discussion of challenges and potential extensions is provided. 
\end{abstract}

\maketitle 
\vspace{-.6cm}
{
\hypersetup{linkcolor=black}
}



\section{Introduction}

Volterra equations are used in a great variety of sciences to model evolution phenomena where memory effects are present in the dynamics. Volterra equations typically take the form
\begin{equation}\label{eq:intro Volterra}
y_t=y_0+\int_0^t k(t,s)f(y_s)dx_s,\quad t\in [0,T]
\end{equation}
where $k$ is called Volterra kernel defined on $[0,T]^2$, and is possibly singular on the diagonal. The process $x$ represents a potentially highly irregular control, such as a stochastic processes or otherwise nowhere differentiable path. In applications, the process $x$ is typically stochastic, or a realized path associated to a stochastic process. 
Thus making sense of the integral appearing in~\eqref{eq:intro Volterra} can be be challenging, and  the construction will often depend highly on the assumptions of both the kernel $k$ and the driving process $x$. Applying It\^o's theory for stochastic differential equations, construction of the integral as well as stochastic well-posedness of the equation has been well established in the case  when the process $x$ is a Brownian motion and the kernel $k$ satisfies the integrability condition $k(t,\cdot)\in L^2([0,t])$ for all $t\in [0,T]$, see e.g. \cite{Zhang,OksZha}. 

Volterra equations has recently received much attention towards modeling of rough volatility. In this case it is desirable to allow $x$ to be a more generic stochastic processes, such as a fractional Brownian motion, while at the same time allowing the Volterra kernel $k$ to be singular, see e.g. \cite{Gath,EuchRosenbaum,Bayer2020} and the references therein. With this application in mind, the authors of  \cite{Bayer2020} observed that for numerical computations related to rough volatility modeling, a pathwise approach to Volterra equations of the form \eqref{eq:intro Volterra} is highly useful.  
 The singular Volterra kernel in combination with the driving Brownian motion created divergence in the covariation $\langle \int_0^{\cdot} k(\cdot,s)dW_s,W\rangle$ appearing as an It\^o-Stratonovich correction, requiring an infinite-type renormalization procedure.    Based on the modern theory of Regularity Structures, the authors proved well-posedness of  equation~\eqref{eq:intro Volterra} under such renormalization.

An alternative approach to deal with Volterra equations in a pathwise manner was proposed in \cite{HarangTindel}. There,  a new generic  methodology based on the theory of rough paths was proposed to treat Banach-valued Volterra equations like \eqref{eq:intro Volterra} in the case when the Volterra kernel $k(t,s)$ is 
 behaving similarly to  $|t-s|^{-\gamma}$ for some $\gamma\geq 0$, and the driving signal $x$ is only assumed to be H\"older continuous (with H\"older regularity possibly lower than $1/2$). 
 The Volterra rough path framework is developed around a splitting of the arguments in a Volterra process, in the sense that one lifts the classical form of Volterra process $z_t:=\int_0^t k(t,s)dx_s$ defined on $[0,T]$, to a two parameter object defined on the simplex $\Delta_2[0,T]:=\{(s,t)\in[0,T]^2 ; \, s\leq t\}$ given formally by 
\begin{equation}\label{eq:simple volterra}
z_t^\tau:=\int_0^t k(\tau,s)dx_s,\quad t\leq \tau. 
\end{equation}
Clearly, when the two parameter object is restricted to the diagonal in $[0,T]^2$, we have  $z_t^t=z_t$, obtaining the classical type of Volterra process. 
The advantage of viewing the Volterra process as this two parameter object is that one can easily distinguish between the regularity contributed by the driving signal versus the possible singularity obtained from the kernel $k$, thus making pathwise regularity analysis easier, and sewing based arguments more straightforward. 

 In a similar spirit as for classical rough paths, the idea is to lift the Volterra signal $(t,\tau)\mapsto z_t^\tau$ as defined in \eqref{eq:simple volterra},  to a signature type object (see e.g. \cite{FriVic,TindelNualart2011}) resembling a collection of iterated integrals, satisfying certain algebraic relations, which is called the Volterra signature.
 In this article we will focus on the second order lift; $z\mapsto (\bfz^1,\bfz^2)$. In the case of a smooth signal $x$, the two components take the form 
 \begin{equation}\label{eq:first rp}
     \bfz^{1,\tau}_{ts}=z^{\tau}_{ts},\qquad \bfz^{2,\tau}_{ts}=\int_s^t k(\tau,r)\int_s^r k(r,u)\,dx_u\,dx_r. 
 \end{equation}
 In contrast to classical rough path theory, the Volterra signature does not satisfy Chen's relation with the tensor product, but a convolution type product is required in order to obtain an equivalent algebraic relation. 
 Indeed, by definition of $\bfz^2$ above, one can readily check that for any $s\leq u\leq t\leq \tau$
 \begin{equation*}
     \bfz^{2,\tau}_{ts}-\bfz^{2,\tau}_{tu}-\bfz^{2,\tau}_{us}\neq \bfz^{1,\tau}_{tu}\otimes \bfz^{1,\tau}_{us}. 
 \end{equation*}
 However, as observed in \cite{HarangTindel} the following generalized Chen's relation holds: 
 \begin{equation*}
     \bfz^{2,\tau}_{ts}-\bfz^{2,\tau}_{tu}-\bfz^{2,\tau}_{us}= \bfz^{1,\tau}_{tu}\ast  \bfz^{1,\cdot}_{us}
 \end{equation*}
 where the convolution product $\ast$ used on the right hand side is defined by 
 \begin{equation}\label{eq:first conv}
     \bfz^{1,\tau}_{tu}\ast  \bfz^{1,\cdot}_{us}:=\lim_{|\cp|\rightarrow 0} \sum_{[u',v']\in\cp } \bfz^{1,\tau}_{v'u'}\otimes  \bfz^{1,u'}_{us}. 
 \end{equation}
 Notice that in the classical rough paths setting where $k(t,s)\equiv 1$, then $z_t^\tau = z_t$, and $\bfz^1\ast \bfz^1=\bfz^1\otimes \bfz^1$. It is the introduction of a Volterra kernel which requires an extension of the classical tensor product in order to obtain a suitable Chen's relation for the Volterra rough path. So far, the assumption has been that $x$ is a smooth path, and in this case both the iterated integral in \eqref{eq:first rp} and the convolution product constructed as an integral in \eqref{eq:first conv} exist by standard integration arguments. However, in the theory of rough paths we are interested in irregular, nowhere differentiable signals $x$, requiring a careful analysis of the construction of these objects.  Once this is in place, 
the Volterra signature in combination with certain controlled Volterra paths is used in~\cite{HarangTindel} to prove existence and uniqueness of  solutions to \eqref{eq:intro Volterra} in a purely pathwise manner. 

Although  \cite{HarangTindel}  provides the basic framework for Volterra rough paths, two important problems relating to this theory was left open:
\begin{description}
    \item[Analytic extension] On the analytic side, \cite{HarangTindel} only deals with the case when  $\alpha-\gamma\geq 1/3$ (where we recall that $\alpha$ is the regularity of the signal, while $\gamma$ is the possible order of singularity from the kernel $k$). To get a complete analytic picture of the framework of Volterra rough paths, this regime must be extended to $\alpha-\gamma>0$. 
    \item[Probabilistic construction] For completeness of the framework it is crucial to provide a complete probabilistic construction of the lift of a stochastic Volterra  process into a Volterra rough path, analogues to the rough path lift for stochastic processes.
\end{description}

Regarding the analytic problem, extending this regime was dealt with in the article \cite{HarangTindelWang}, where the algebraic framework was described for $\alpha-\gamma\geq 1/4$. In a very recent article \cite{bruned2021ramification}, Bruned and Kastetsiadis  extends this even further to all $\alpha-\gamma>0$ by invoking algebraic theories similar to that used for non-geometric rough paths \cite{Gubinelli2010,HK2015} and  regularity structures \cite{Hairer2014}. 

The problem of a probabilistic construction of the Volterra rough path is the main goal of the current article. More specifically, our main contribution is twofold:
\begin{enumerate}[wide, labelwidth=!, labelindent=0pt, label=(\roman*)]
\setlength\itemsep{.03in}

\item
 As the framework for Volterra rough paths relies on spaces for Volterra--H\"older  paths with two parameters (one corresponding to regularity and one to singularity), a direct application of the classical Kolmogorov continuity theorem will not provide sufficient answers. Hence new arguments need to be developed, specifically suited for the type of H\"older spaces necessary to properly define rough Volterra equations. Extending the Garsia--Rodemich--Rumsey (GRR) inequality to suit Volterra paths is therefore the first aim of this article.
 This extension is not only highly useful for the probabilistic treatment in the context of Volterra rough paths, but could also prove valuable towards applications for other types of singular H\"older norms, such as those considered in \cite{Bellingeri2021SingularPS}.  
 
 \item
 With the singular GRR inequality in hand, we provide a construction of the Volterra rough path in the regime $\alpha-\gamma\geq \frac{1}{3}$ (requiring one iterated integral) by using tools from the theory of Malliavin calculus. This construction is both done in the case when the driving stochastic process is a fractional Brownian motion with $H>\frac{1}{2}$ together with a singular kernel. We also handle the case of classical Brownian motion with a singular kernel. Note that in both  cases the construction of a Volterra rough path like \eqref{eq:first rp} is required. Indeed, a singular kernel behaving similarly to $|t-s|^{-\gamma}$ pushes down the regularity of the Volterra process constructed from an fBm to be $H-\gamma$. This exponent can be smaller than $\frac{1}{2}$ even though $H>\frac{1}{2}$. This is in contrast to the classical rough path regime, where the rough path associated to a fractional Brownian motion with $H>\frac{1}{2}$ can simply be constructed by classical Young theory.  

\end{enumerate}
\noindent
As the reader can see from the description above, our analysis will be a delicate combination of analytic and probabilistic techniques.

The article is organized as follows: Section \ref{sec:prelim} provides an overview of Volterra paths, Volterra-H\"older spaces, as well as a summary of central concepts from \cite{HarangTindel} regarding the convolution product and Volterra sewing. In Section \ref{fbm case} the extension of the GRR inequality is provided. Section \ref{Volterra rough path driven by fractional Brownian motion} deals with the construction of the Volterra rough path for fractional Brownian motion with $H>\frac{1}{2}$. In Section \ref{sec:VRP for BM} this construction is  extended to the case of a regular Brownian motion. A discussion on further extensions to rough fractional Brownian motion is discussed in the end of Section \ref{sec:VRP for BM}.

\section{Preliminary results}\label{sec:prelim}
In \cite{HarangTindel} and \cite{HarangTindelWang}, the Volterra rough formalism was based on certain spaces of functions having specific regularity/singularity features. Before defining the proper spaces quantifying this type of regularity, let us introduce some notation:
\begin{notation}\label{notation of delta}
Let $T>0$ be a finite time horizon, and $n\geq 2$. Then the simplex $\Delta_{n}^{T}$ is defined by 
\[
\Delta_{n}^{T}\big\{(s_{1},\ldots, s_{n})\in [0,T]^{n}; 0\leq s_{1}<\cdots<s_{n}\leq T\big\}.
\]
When this causes no ambiguity, we will abbreviate $\Delta_{n}^{T}$ as $\Delta_{n}$. For $(s,t)\in \Delta_{2}$, we designate $\mathcal{P}$ to be a generic partition of $[s,t]$. Two successive points forming an interval contained in this partition are written as $[u,v]\in\mathcal{P}$.
\end{notation}
The functions quantifying our regularities are also labeled in the following notation.
\begin{notation}\label{notation of psi}
Consider  four parameters $\alpha, \gamma\in (0,1)$ and $\zeta, \eta\in [0,1]$ satisfying 
\begin{eqnarray}\label{condition parameters}
\rho\equiv \alpha-\gamma>0, \qquad 0\le \zeta \le \inf(\rho,\eta).
\end{eqnarray}
 For $(s,t,\tau^{\prime}, \tau)\in \Delta_4$, we set 
\begin{equation}\label{psi 1}
\psi^{1}_{(\alpha,\gamma)}(\tau,t,s)=\lc |\tau-t|^{-\gamma}|t-s|^{\alpha}\rc\wedge |t-s|^{\rho}, 
\end{equation}
and 
\begin{equation}\label{psi 12}
\psi^{1,2}_{(\alpha,\gamma,\eta,\zeta)}(\tau,\tau^{\prime},t,s)=|\tau-\tau^{\prime}|^{\eta}|\tau^{\prime}-t|^{-(\eta-\zeta)}\left(\lc |\tau^{\prime}-t|^{-\gamma-\zeta}|t-s|^{\alpha} \rc\wedge |t-s|^{\rho-\zeta}\right).
\end{equation}
\end{notation}
In the next defintion the functional spaces called $\mathcal{V}^{(\alpha,\gamma,\eta,\zeta)}$ are introduced, which are equivalent to those used in \cite{HarangTindel, HarangTindelWang}. As is evident from the analysis in \cite{HarangTindel, HarangTindelWang}, these spaces are natural function sets when dealing with Volterra type regularities.
\begin{Def}\label{Volterra space}
Consider four parameters $\alpha, \gamma\in (0,1)$ and $\zeta, \eta\in [0,1]$ satisfying relation~\eqref{condition parameters}, and fix $m\geq1$.  Throughout the article we consider functions  $z:\Delta_{2}\rightarrow \mathbb{R}^{m}$ of the form $(t,\tau)\mapsto z_{t}^{\tau}$, such that $z_{0}^{\tau}=z_{0}$ for all $\tau\in(0,T]$. We define the space of Volterra paths of index $(\alpha,\gamma,\eta,\zeta)$, denoted by $\mathcal{V}^{(\alpha,\gamma,\eta,\zeta)}(\Delta_{2};\mathbb{R}^{m})$, as the set of such functions satisfying  
\begin{equation}\label{Volterra norm}
\|z\|_{(\alpha,\gamma,\eta,\zeta)}=\|z\|_{(\alpha,\gamma),1}+\|z\|_{(\alpha,\gamma,\eta,\zeta),1,2} < \infty.
\end{equation}
Recalling Notation \ref{notation of delta} and \ref{notation of psi},
the 1-norms and 1,2-norms in~\eqref{Volterra norm} are respectively defined as follows:
\begin{flalign}\label{V norm}
\| z\|_{\left(\alpha,\gamma\right),1}&=\sup_{\left(s,t,\tau\right)\in\Delta_{3}}
\frac{|z_{ts}^{\tau}|}{\psi^{1}_{(\alpha,\gamma)}(\tau,t,s)},\\
\label{V norm b}
\| z\|_{\left(\alpha,\gamma,\eta,\zeta\right),1,2}&
=\sup_{\left(s,t,\tau^{\prime},\tau\right)\in\Delta_{4}}\frac{|z_{ts}^{\tau\tau^{\prime}}|}{\psi^{1,2}_{(\alpha,\gamma,\eta,\zeta)}(\tau,\tau^{\prime},t,s)},
\end{flalign}
with the convention $z^{\tau}_{ts}=z^{\tau}_{t}-z^{\tau}_{s}$ and $z^{\tau\tau^{\prime}}_{s}=z^{\tau}_{s}-z^{\tau^{\prime}}_{s}$. 
Notice that under the mapping 
\[
z\mapsto|z_{0}|+\| z\|_{\left(\alpha,\gamma,\eta,\zeta\right)},
\]
 the space $\mathcal{V}^{\left(\alpha,\gamma,\eta,\zeta\right)}$ is a Banach
space. 
\end{Def}
\begin{remark}\label{volterra embedding}
As mentioned in \cite[Remark 2.6]{HarangTindelWang}, the spaces $\mathcal{V}^{(\alpha,\gamma,\eta,\zeta)}$ enjoy embedding properties of the form $\mathcal{V}^{(\alpha,\gamma,\eta,\zeta)}\subset \mathcal{V}^{(\beta,\gamma,\eta,\zeta)}$ for $0<\alpha<\beta<1$. In addition, the norms defined by \eqref{Volterra norm}-\eqref{V norm b} verify the following relation on $[0,T]$:
\begin{equation*}
\|y\|_{(\beta,\gamma),1}\leq T^{\alpha-\beta}\|y\|_{(\alpha,\gamma),1},\quad
\|y\|_{(\beta,\gamma,\eta,\zeta),1,2}\leq T^{\alpha-\beta}\|y\|_{(\alpha,\gamma,\eta,\zeta),1,2},\quad
\|y\|_{(\beta,\gamma,\eta,\zeta)}\leq T^{\alpha-\beta}\|y\|_{(\alpha,\gamma,\eta,\zeta)}.
\end{equation*}
\end{remark}

\begin{remark}
The spaces given in Definition \ref{Volterra space} are slightly different than the once introduced in \cite{HarangTindel}, as there is now supremum over the parameters $\eta$ and $\zeta$ appearing here. It was observed in \cite{HarangTindelWang} that the supremum was unnecessary for the Volterra rough path methodology to work, but one must instead introduce an assumption that the Volterra paths of interest is contained in a suitable family Volterra spaces as those in Definition \ref{Volterra space}.
Avoiding the original supremum also makes probabilistic analysis, as we will consider in the subsequent sections, more tractable. 
We therefore use here the same types of norms and spaces here as the ones introduced in \cite{HarangTindelWang} 
\end{remark}

\begin{remark}
Comparing \eqref{psi 1} and \eqref{psi 12}, one can relate the functions $\psi^{1}$ and $\psi^{1,2}$ in the following way: 
\begin{equation}\label{relation psi 1 and 12}
\psi^{1,2}_{(\alpha,\gamma,\eta,\zeta)}(\tau,\tau^{\prime},t,s)=\left|\tau-\tau^{\prime}\right|^{\eta}\left|\tau^{\prime}-t\right|^{-(\eta-\zeta)}\psi^{1}_{(\alpha,\gamma+\zeta)}(\tau^{\prime},t,s)
\end{equation}
As explained in \cite[Proposition 2.10]{HarangTindelWang}, the parameter $\eta$ above accounts for the regularity of a Volterra path in the upper variables $\tau$, $\tau^{\prime}$. Then one plays with extra parameters $\zeta$ in order to get regularities for paths of the form $r\mapsto z^{r}_r$.
\end{remark}

As illustrated in the introduction, convolution products plays a central role for the subsequent  considerations of the Volterra rough path.  Let us recall a proposition from \cite{HarangTindelWang} giving explicit meaning  to this concept, and establishing the existence in a general setting.
\begin{prop}\label{one step conv}
We consider two Volterra paths
 $z\in\mathcal{V}^{\left(\alpha,\gamma,\eta,\zeta\right)}(\mathbb{R}^{m})$ and $y\in\mathcal{V}^{(\alpha,\gamma,\eta,\zeta)}(\mathcal{L}(\RR^{m}))$ as given in Definition \ref{Volterra space}. On top of condition \eqref{condition parameters}, we assume that the exponents $\alpha, \eta$ are such that $\eta>1-\alpha$. Otherwise stated, our parameters $\alpha, \gamma, \zeta, \eta$ satisfy
 \begin{eqnarray}\label{condition parameters 2}
\rho\equiv \alpha-\gamma>0, \qquad 0\le \zeta\le \inf(\rho, \eta), \quad\text{and}\quad \eta>1-\alpha .
\end{eqnarray}
Then the convolution product of the two Volterra paths $y$ and $z$ is a bilinear operation on $\mathcal{V}^{\left(\alpha,\gamma,\eta,\zeta\right)}(\mathbb{ R}^{m})$ given by 
\begin{equation}\label{convolution in 1d}
\text{\ensuremath{z_{tu}^{\tau}\ast y_{us}^{\cdot}}}
=
\int_{t>r>u}dz_{r}^{\tau} y_{us}^{r} :=\lim_{\left|\mathcal{P}\right|\rightarrow0}\sum_{\left[u^{\prime},v^{\prime}\right]\in\mathcal{P}} z_{v^{\prime}u^{\prime}}^{\tau} y_{us}^{u^{\prime}},
\end{equation}
where $\mathcal{P}$ is a generic partition of $[u,t]$ for which we recall Notation \ref{notation of delta}. 
 The integral in \eqref{convolution in 1d} is understood as a Volterra-Young integral for all $ (s,u,t,\tau)\in \Delta_{4}$. Moreover, the following two inequalities hold for  any tuple $(s,u,t,\tau,\tau')$ lying in $\Delta_5$:
\begin{align}\label{eq:z conv y bound}
\left|z_{tu}^{\tau}\ast y_{us}^{\cdot}\right|&\lesssim\|z\|_{\left(\alpha,\gamma\right),1}\|y\|_{\left(\alpha,\gamma,\eta,\zeta\right),1,2} \, \psi^{1}_{(2\rho+\gamma,\gamma)}(\tau,t,s),
\\
\left|z_{tu}^{\tau'\tau}\ast y_{us}^{\cdot}\right|&\lesssim\|z\|_{\left(\alpha,\gamma,\eta,\zeta\right),1,2}\|y\|_{\left(\alpha,\gamma,\eta,\zeta\right),1,2} \, \psi^{1,2}_{(2\rho+\gamma, \gamma,\eta,\zeta)}(\tau,\tau^{\prime},t,s). \label{eq:z conv y bound 2}
\end{align}
\end{prop}

Commonly for rough path based theories, we will also here work with the $\delta$--operator. The next notation recalls this operator which will be significant for subsequent proofs.   

\begin{notation}\label{notation of delta g}
Let $g$ be a path from $\Delta_{2}$ to $\mathbb{R}^{m}$, and consider $(s,u,t)\in \Delta_{3}$. Then the quantity $\delta_{u}g_{ts}$ is defined by 
\begin{equation}\label{delta}
\delta_{u}g_{ts}=g_{ts}-g_{tu}-g_{us}.
\end{equation}
\end{notation}

With Definition \ref{Volterra space} and Proposition \ref{one step conv} in hand we are now ready to state the main assumption used in \cite{HarangTindelWang}. Namely the  Volterra rough paths analysis relies on the ability to construct a family $\{z^{j,\tau}; j\leq n\}$ of Volterra iterated integrals according to the following definition: 

\begin{Def}\label{hyp 2}
Consider $\alpha, \gamma\in (0,1)$,  and for an arbitrary finite integer $N\ge 1$, let $\{\zeta_{k},\eta_{k}; 1\le k\le N\}$ be a family of exponents satisfying the relation \eqref{condition parameters 2}. Then for $n=\lfloor \rho^{-1}\rfloor $, this family $\{\bfz^{j,\tau}; j\leq n\}$  is assumed to enjoy the following properties: 
\begin{enumerate}[wide, labelwidth=!, labelindent=0pt, label=(\roman*)]

\item
$\bfz^{1}=z$ and $\bfz^{j,\tau}_{ts}\in (\mathbb{R}^{m})^{\otimes j}$.

\item
For all $j\leq n$ and $(s,t,\tau)\in \Delta_{3}$ we have 
\begin{equation}\label{hyp b}
\delta_{u}\bfz^{j,\tau}_{ts}=\sum_{i=1}^{j-1}\bfz^{j-i,\tau}_{tu}\ast \bfz^{i,\cdot}_{us}=\int_{s}^{t}d\bfz^{j-i,\tau}_{tr}\, \bfz^{i,r}_{us},
\end{equation}
where the right hand side of \eqref{hyp b} is given by Proposition \ref{one step conv}.
\item\label{zj space}
For all $j=1,\ldots,n$, we have $ {\bf z}^{j}\in \bigcap^{N}_{k=1}\mathcal{V}^{(j\rho+\gamma,\gamma,\eta_{k},\zeta_{k})}$.
\end{enumerate}
\end{Def}

As the reader might have observed,  Definition \ref{hyp 2} is a natural extension of the more classical definition of rough path, see e.g. \cite{FriHai}, where convolution products naturally extends the tensor product for Volterra paths. In the decomposition \eqref{hyp b}, it is desirable to  quantify the regularity of the objects depending on the variables $(s,u,t,\tau)\in \Delta_{4}$. Thus a small variation of Definition~\ref{Volterra space} is suitable for this quantification, illustrated in the next definition (see also in \cite[Definition 2.9]{HarangTindelWang}).

\begin{Def}\label{def of delta norm}
As in Definition \ref{Volterra space}, consider $m\geq 1$, as well as four parameters $\alpha,\gamma\in (0,1)$, $\eta, \zeta\in [0,1]$ satisfying the relation~\eqref{condition parameters 2}. Let $\bfz: \Delta_{4}\rightarrow \mathbb{R}^{m}$ be of the form $(s,u,t,\tau)\mapsto \bfz^{\tau}_{tus}$. The definition of $\mathcal{V}^{(\alpha,\gamma,\eta,\zeta)}(\Delta_{3};\mathbb{R}^{m})$ can be extended in order to define a space $\mathcal{V}^{(\alpha,\gamma,\eta,\zeta)}(\Delta_{4};\mathbb{R}^{m})$, by using the same definition as \eqref{Volterra norm}. That is we have $z\in\mathcal{V}^{(\alpha,\gamma,\eta,\zeta)}(\Delta_{4};\mathbb{R}^{m})$ if
\begin{equation}\label{Volterra norm delta}
\|z\|_{(\alpha,\gamma,\eta,\zeta)}=\|z\|_{(\alpha,\gamma),1}+\|z\|_{(\alpha,\gamma,\eta,\zeta),1,2} < \infty.
\end{equation}
The quantities $\|z\|_{(\alpha,\gamma),1}$ and $\|z\|_{(\alpha,\gamma,\eta,\zeta),1,2}$ in \eqref{Volterra norm delta} are slight modifications of \eqref{V norm} and \eqref{V norm b}, respectively defined by
\begin{equation}\label{delta zj norm 1}
 \|\bfz\|_{(\alpha,\gamma),1}=\sup_{(s,u,t,\tau)\in\Delta_{4}}\frac{|\bfz_{tus}^{\tau}|}{\psi^{1}_{(\alpha,\gamma)}(\tau,t,s)},
 \end{equation} 
and 
 \begin{equation}\label{delta zj norm 12}
 \|\bfz\|_{(\alpha,\gamma,\eta,\zeta),1,2}=\sup_{\left(s,u,t,\tau^{\prime},\tau\right)\in\Delta_{5}}\frac{|\bfz^{\tau\tau^{\prime}}_{tus}|}{\psi^{1,2}_{(\alpha,\gamma,\eta,\zeta)}(\tau,\tau^{\prime},t,s)}.
 \end{equation}
\end{Def}

\section{An extension of Garsia-Rodemich-Rumsey's inequality}\label{fbm case}

This section is devoted to extend Garsia-Rodemich-Rumsey's celebrated result \cite{Garsia} to the Volterra space $\mathcal{V}^{(\alpha,\gamma,\eta,\zeta)}$ introduced in Definition \ref{Volterra space}. To this aim, we introduce two integral functionals resembling the role of a Sobolev norm, tailored for the regularity functions introduced in \eqref{psi 1}--\eqref{psi 12}. These will be used to extend the Garsia--Rodemich--Rumsey inequality to Volterra paths.  
\begin{Def}\label{def of U volterra}
Let $z:\Delta_3\rightarrow \RR^d$ be a continuous  Volterra increment. Then for some parameters  $p\geq1$ and  $\alpha,\gamma\in (0,1)$, $\eta, \zeta\in [0,1]$ satisfy the relation \eqref{condition parameters} we define
\begin{align}
  U^{\tau}_{(\alpha,\gamma),p,1}\left(z;\eta,\zeta\right)&:=\left(\int_{(v,w)\in \Delta^{\tau}_{2}}\frac{|z^{\tau}_{wv}|^{2p}}{|\tau-w|^{-2p(\eta-\zeta)}|\psi^{1}_{(\alpha,\gamma+\zeta)}(\tau,w,v)|^{2p}|w-v|^{2}}dvdw \right)^{\frac{1}{2p}}\label{U1}
    \\
    U^{\tau}_{(\alpha,\gamma,\eta,\zeta),p,1,2}\left(z\right)&:=\left(\int_{(v,w,r^{\prime},r)\in\Delta^{\tau}_{4}}\frac{|z^{rr'}_{wv}|^{2p}}{|\psi^{1,2}_{(\alpha,\gamma,\eta,\zeta)}(r,r',w,v)|^{2p}|w-v|^{2}|r-r'|^2}dvdw dr^{\prime}dr\right)^{\frac{1}{2p}}\label{U12},
\end{align}
where recall that the functions $\psi^{1}, \psi^{1,2}$ are respectively defined in \eqref{psi 1} and \eqref{psi 12}.
\end{Def}
\begin{remark}\label{rem:increasing}
Notice that if we set 
\begin{equation*}
D^{\tau}(w,v)=\frac{|z^{\tau}_{wv}|^{2p}}{|\tau-w|^{-2p(\eta-\zeta)}|\psi^{1}_{(\alpha,\gamma)}(\tau,w,v)|^{2p}|w-v|^{2}},
\end{equation*}
then we trivially have $D^{\tau}(w,v)\ge 0$. Plugging this information in relation \eqref{U1}, we get that $\tau\mapsto U^{\tau}_{(\alpha,\gamma), p,1}(z;\eta,\zeta)$ is a non-decreasing function. Thus for $\tau\le T$ we have
$U^\tau_{(\alpha,\gamma),p,1}(z;\eta,\zeta)\le U^T_{(\alpha,\gamma),p,1}(z;\eta,\zeta)$. 
\end{remark}
\begin{remark}
The quantity $U^{\tau}_{(\alpha,\gamma),p,1}\left(z;\eta,\zeta\right)$ evaluated at $\eta=\zeta=0$, will be denoted by $U^{\tau}_{(\alpha,\gamma),p,1}\left(z\right)$ for notational sake.
\end{remark}

Using the above integral functionals, we now state and prove the extension of  Garsia-Rodemich-Rumsey's inequality for general Volterra increments on $\Delta_3$. This will in turn be applied to provide an upper bound for the Volterra norms introduced in Definition \ref{Volterra space} in terms of the integral functionals in Definition \ref{def of U volterra}.

\begin{lemma}\label{GRR lemma volterra}
Let $\bfz:\Delta_3\rightarrow \RR^d$ be a continuous increment.  Consider 4 parameters $\kappa,  \gamma, \eta, \zeta$ such that 
\begin{equation}\label{relation kappa}
0\le \gamma<\kappa<1, \quad 0\le \zeta<\kappa-\gamma, \quad\text{and}\quad \zeta\le \eta\le1.
\end{equation}
Then there  exists a universal constant $C>0$ such that for all $(s,t,\tau)\in \Delta_3$ we have
\begin{eqnarray}
|\bfz_{ts}^\tau|\leq C |\tau-t|^{-(\eta-\zeta)}\psi^1_{\kappa,\gamma+\zeta}(\tau,t,s)\left(U^{\tau}_{(\kappa,\gamma),p,1}(\bfz;\eta,\zeta)+ \|\delta \bfz\|^{[s,t]}_{(\kappa,\gamma,\eta,\zeta),1}\right),\label{U1 bound}
\end{eqnarray}
where the quantity $\|\delta \bfz\|^{[s,t]}_{(\kappa,\gamma,\eta,\zeta),1}$ is defined as 
\begin{equation}\label{delta z in [s,t]}
\|\delta \bfz\|^{[s,t]}_{(\kappa,\gamma,\eta,\zeta),1}=
\sup_{s\le u<v\le t} 
\frac{|\delta _u\bfz^\tau_{vs}|}{|\tau-v|^{-(\eta-\zeta)}\psi_{(\kappa,\gamma+\zeta)}^1(\tau,v,s)}.
\end{equation}
In particular, for $\eta=\zeta=0$, we have 
\begin{equation}\label{eq:eta=zeta=0 bound}
    |\bfz_{ts}^\tau|\lesssim \psi^{1}_{(\kappa,\gamma)}(\tau,t,s)\left(U^{\tau}_{(\kappa,\gamma),p,1}(\bfz)+ \|\delta \bfz\|_{(\alpha,\gamma),1}\right),\end{equation} 
where $ \|\delta \bfz\|_{(\alpha,\gamma),1}$ is given by \eqref{delta zj norm 1}.
\end{lemma}

\begin{proof}
Consider a tuple $(s,t,\tau)\in \Delta_{3}$, with $t-s<\frac{T}{2}$. First construct a sequence of points $(s_{k})_{k\geq 0}$, such that $s_{k}\in[0,T]$ and $s_{k}$ convergs to $s$ by induction. Namely, set $s_{0}=t$, and suppose that $s_{0},s_{1},\ldots, s_{k}$ have been constructed, and let $D_{k}=(s, \frac{s_{k}+s}{2})$. Define the  function $I$ as follows: 
\begin{equation}\label{define of Iv}
I(w):=\int_{s}^{w}\frac{|\bfz^{\tau}_{wv}|^{2p}}{|\tau-w|^{-2p(\eta-\zeta)}|\psi^{1}_{(\kappa,\gamma+\zeta)}(\tau,w,v)|^{2p}|w-v|^{2}}dv.
\end{equation}
According to the value of $I$, define two subsets of the interval $D_{k}$:
\begin{eqnarray}
A_{k}&:=&\left\{w\in D_{k}\Big|\quad I(w)>\frac{4(U^{\tau}_{(\kappa,\gamma),p,1}\left(\bfz;\eta,\zeta\right))^{2p}}{|s_{k}-s|}\right\},
\label{Ak volterra}\\
B_{k}&:=&\left\{w\in D_{k}\Big|\quad \frac{|\bfz^{\tau}_{s_{k}w}|^{2p}}{|\tau-s_k|^{-2p(\eta-\zeta)}|\psi^{1}_{(\kappa,\gamma+\zeta)}(\tau,s_{k},w)|^{2p}|s_{k}-w|^{2}}>\frac{4I(s_{k})}{|s_{k}-s|}\right\},
\label{Bk volterra}
\end{eqnarray}
where we recall again that $\psi^{1}_{(\kappa, \gamma+\zeta)}(\tau,w,v)$ is given by \eqref{psi 1}. We claim that $A_{k}\cup B_{k}\subset D_{k}$, where the inclusion is strict. Toward proving this claim, observe that the set of $(v,w)$ such that
\[
s<v<w<\frac{s_{k}+s}{2}
\]
is included in $[0,T)^{2}$. Hence due to the definition \eqref{U1} of $U^{\tau}_{(\kappa,\gamma),p,1}(\bfz;\eta,\zeta)$ we get
\begin{equation}\label{relation between I and A Volterra}
\left(U^{\tau}_{(\kappa,\gamma),p,1}\left(\bfz;\eta,\zeta\right)\right)^{2p}\geq\int_{A_{k}}dv\,I(v).
\end{equation}
Therefore thanks to relation \eqref{Ak volterra} defining $A_{k}$, we get
\begin{equation}\label{relation U and Ak Volterra}
\left(U^{\tau}_{(\kappa,\gamma),p,1}(\bfz;\eta,\zeta)\right)^{2p}>\frac{4\left(U_{(\kappa,\gamma),p,1}^{\tau}(\bfz;\eta,\zeta)\right)^{2p}}{|s_{k}-s|}\mu(A_{k}),
\end{equation}
where $\mu(A_{k})$ denotes the Lebesgue measure of set $A_{k}$.
It is thus readily checked from \eqref{relation U and Ak Volterra} that 
\begin{equation}\label{upper bound for Ak Volterra}
\mu(A_{k})<\frac{|s_{k}-s|}{4}=\frac{\mu{(D_{k})}}{2} .
\end{equation}
Arguing similarly for the set $B_{k}$,  we  note that since the set $B_{k}$ defined by \eqref{Bk volterra} is a subset of $(s,s_{k})$, we have 
\begin{equation}\label{relation between Bk and I Volterra}
I(s_{k})\geq\int_{B_{k}}\frac{|\bfz^{\tau}_{s_{k}v}|^{2p}}{|\tau-s_k|^{-2p(\eta-\zeta)}|\psi^{1}_{(\kappa,\gamma+\zeta)}(\tau,s_{k},v)|^{2p}|s_{k}-v|^{2}}dv.
\end{equation}
Thus plugging the definition \eqref{Bk volterra} of $B_{k}$ in the right hand side of \eqref{relation between Bk and I Volterra}, we get
\begin{equation*}
I(s_{k})>\frac{4\mu(B_{k})}{|s_{k}-s|}I(s_{k}),
\end{equation*}
from which we obtain again that 
\begin{equation}\label{upper bound for Bk Volterra}
\mu(B_{k})<\frac{|s_{k}-s|}{4}=\frac{\mu{(D_{k})}}{2}
\end{equation}
Combining \eqref{upper bound for Ak Volterra} and \eqref{upper bound for Bk Volterra}, we have thus obtained
\begin{equation*}
\mu(A_{k})<\frac{\mu{(D_{k})}}{2},\quad\text{and}\quad\mu(B_{k})<\frac{\mu{(D_{k})}}{2},
\end{equation*}
and it follows that
\begin{equation}\label{relation A,B and D Volterra}
\mu(A_{k})+\mu(B_{k})<\mu(D_{k}),
\end{equation}
from which we easily deduce that $A_{k}\cup B_{k}$ is a strict subset of $D_{k}$.
Now we can choose $s_{k+1}$ arbitrarily in $D_{k} \setminus (A_{k}\cup B_{k})$. Summarizing our considerations so far; for all $n$ we have  constructed a family $\{s_{0},\ldots,s_{n}\}$ such that for all $0\le k \le n$, we have $0\le s_{k}-s \le \frac{t-s}{2^{k}}$ and the following two conditions are met:
\begin{eqnarray}\label{a1 Volterra}
 \frac{|\bfz^{\tau}_{s_{k}s_{k+1}}|^{2p}}{|\tau-s_k|^{-2p(\eta-\zeta)}|\psi^{1}_{(\kappa,\gamma+\zeta)}(\tau,s_{k},s_{k+1})|^{2p}|s_{k}-s_{k+1}|^{2}} 
 &\le& 
 \frac{4I(s_{k})}{|s_{k}-s|},
\\ 
I(s_{k+1}) 
&\le& 
\frac{4\left(U_{(\kappa,\gamma),p,1}^{\tau}(\bfz;\eta,\zeta)\right)^{2p}}{|s_{k}-s|}
. \notag
\end{eqnarray}
With \eqref{a1 Volterra} in hand,  decompose $\bfz^{\tau}_{ts}$ into 
\begin{equation}\label{decomposition of zj}
\bfz^{\tau}_{ts}=\bfz^{\tau}_{s_{n+1}s}+\sum_{k=0}^{n}\left(\bfz^{\tau}_{s_{k}s_{k+1}}+\delta_{s_{k+1}} \bfz^{\tau}_{s_{k}s}\right).
\end{equation}
The aim is now to  bound the term $\bfz^{\tau}_{s_{k}s_{k+1}}$ in \eqref{decomposition of zj}. To this aim, notice that since $s_{k+1}\not\in B_{k}$, we have 
\begin{equation}\label{relation zj and I}
\frac{|\bfz^{\tau}_{s_{k}s_{k+1}}|^{2p}}{|\tau-s_k|^{-2p(\eta-\zeta)}|\psi^{1}_{(\kappa,\gamma)}(\tau,s_{k},s_{k+1})|^{2p}|s_{k}-s_{k+1}|^{2}} \le 4\frac{I(s_{k})}{|s_{k}-s|}.
\end{equation}
Moreover, we also have $s_{k}\notin A_{k-1}$. Hence we obtain
\begin{equation}\label{relation I and U Volterra} 
I(s_{k})<\frac{4\left(U^{\tau}_{(\kappa,\gamma),p,1}(\bfz;\eta,\zeta)\right)^{2p}}{|s_{k-1}-s|}.
\end{equation}
Gathering \eqref{relation zj and I} and \eqref{relation I and U Volterra} yields 
\begin{multline}\label{225 Volterra}
\frac{|\bfz^{\tau}_{s_{k}s_{k+1}}|^{2p}}{|\tau-s_k|^{-2p(\eta-\zeta)}|\psi^{1}_{(\kappa,\gamma+\zeta)}(\tau,s_{k},s_{k+1})|^{2p}|s_{k}-s_{k+1}|^{2}}
\\
<\frac{16\left(U^{\tau}_{(\kappa,\gamma),p,1}(\bfz;\eta,\zeta)\right)^{2p}}{|s_{k}-s||s_{k-1}-s|}
\lesssim \frac{\left(U^{\tau}_{(\kappa,\gamma),p,1}(\bfz;\eta,\zeta)\right)^{2p}}{|s_{k}-s|^{2}},
\end{multline}
where we have used the fact that $|s_{k}-s|\leq|s_{k-1}-s|$ for the second inequality. In addition, thanks to $|s_{k}-s_{k+1}|\leq|s_k-s|$, it is easily seen that we can recast \eqref{225 Volterra} as
\begin{flalign}\label{upper bound for zj}
|\bfz^{\tau}_{s_{k}s_{k+1}}|&\lesssim U^{\tau}_{(\kappa,\gamma),p,1}(\bfz;\eta,\zeta)\, \psi^{1}_{(\kappa,\gamma+\zeta)}(\tau,s_{k},s) |\tau-s_k|^{-(\eta-\zeta)}.
\end{flalign}
Next recall that $\eta$ is assumed to be larger than $\zeta$, and in addition we assume that $0\le \zeta<\kappa-\gamma$. Thus owing to the fact that $|\tau-s_{k}|\ge |\tau-t|$, $|s_{k}-s|\lesssim2^{-k}(t-s)$, and recalling the expression~\eqref{psi 1} for $\psi^{1}$, we end up with 
\begin{align*}
\left|\bfz^{\tau}_{s_{k}s_{k+1}}\right|&\lesssim \frac{U^{\tau}_{(\kappa,\gamma),p,1}(\bfz;\eta,\zeta)}{2^{k(\kappa-\gamma-\zeta)}}\psi^{1}_{(\kappa,\gamma+\zeta)}(\tau,t,s) |\tau-t|^{-(\eta-\zeta)}
\end{align*}
Summing this inequality over $k$ (and using that $\kappa-\gamma-\zeta>0$), we get the following bound for the right hand side of~\eqref{decomposition of zj}:
\begin{equation}\label{upper bound for sum zj}
\left|\sum_{k=0}^{n}\bfz^{\tau}_{s_{k}s_{k+1}}\right|\lesssim U^{\tau}_{(\kappa,\gamma),p,1}(\bfz;\eta,\zeta)\,\psi^{1}_{(\kappa,\gamma+\zeta)}(\tau,t,s)|\tau-t|^{-(\eta-\zeta)}.
\end{equation}
Now we turn to bound the second term $\delta_{s_{k+1}} \bfz^{\tau}_{s_{k}s}$ in the right hand side of \eqref{decomposition of zj}. It is clear that  
\[
|\delta_{s_{k+1}} \bfz^{\tau}_{s_{k}s}|\lesssim \|\delta \bfz\|^{[s,t]}_{(\kappa,\gamma,\eta,\zeta),1}\,|\tau-t|^{-(\eta-\zeta)}\psi_{(\kappa,\gamma+\zeta)}^1(\tau,s_k,s),
\]
recalling that $\|\delta \bfz\|^{[s,t]}_{(\kappa,\gamma,\eta,\zeta),1}$ as given in \eqref{delta z in [s,t]}.
Hence similarly to \eqref{upper bound for sum zj}, we obtain
\begin{equation}\label{upper bound for sum delta zj}
|\sum_{k=0}^{n}\delta_{s_{k+1}} \bfz^{\tau}_{s_{k}s}|\lesssim \|\delta \bfz\|^{[s,t]}_{(\kappa,\gamma,\eta,\zeta),1}\,|\tau-t|^{-(\eta-\zeta)}\psi_{(\kappa,\gamma+\zeta)}^1(\tau,t,s).
\end{equation}
Plugging \eqref{upper bound for sum zj} and \eqref{upper bound for sum delta zj} into \eqref{decomposition of zj}, and letting $n\to \infty$, we get 
 relation~\eqref{U1 bound} thanks to the continuity of $\bfz$. This completes the proof.
\end{proof}
In preparation for the next proposition, we recall here a classical Sobolev embedding inequality. The particular form of the inequality stated here is  as a consequence of the classical  Garsia-Rodemich-Rumsey inequality \cite{GarsiaRodemichRumsey1970}, and can be found stated in the form below in \cite[pp. 2]{LeHu2013}. 
\begin{prop}\label{prop:GRR}
Let $h:[a,b]\rightarrow \RR^d$ be continuous. Then for any $p>\frac{1}{\alpha}$ the following inequality holds
\begin{equation}
    |h_{ts}|\lesssim_{\alpha,p} |t-s|^{\alpha}\left(\int_a^b\int_a^u \frac{|h_{uv}|^p}{|u-v|^{2+p\alpha}}dvdu\right)^{\frac{1}{p}}, 
\end{equation}
where we have set $h_{ts}=h_{t}-h_{s}$ for $(s,t)\in\Delta_{2}$.
\end{prop}

We follow up with a technical lemma, combining Proposition \ref{prop:GRR} with Lemma \ref{GRR lemma volterra}. 
\begin{lemma}\label{lem: rec inc bound}
Let $\bfz:\Delta_3\rightarrow \RR^d$ be continuous. Consider four parameters $\alpha,\gamma\in (0,1)$, $\eta, \zeta\in [0,1]$ that satisfy the relation \eqref{condition parameters}. Recall that $\psi^{1,2}$ is defined by \eqref{psi 12} and the quantities $U$ are introduced in Definition \ref{U12}. Then for any $\alpha-\gamma>\frac{1}{p}$, the following inequality holds for any $(s,t,\tau^{\prime},\tau)\in \Delta^{T}_4,$ 
\begin{multline}
\label{bound}
 \left(  \frac{|\bfz^{\tau\tau'}_{ts}|}{\psi^{1,2 }_{\alpha,\gamma,\eta,\zeta}(\tau,\tau',t,s)}\right)^{2p}
  \lesssim  
    U^T_{(\alpha,\gamma,\eta,\zeta),p,1,2}(\bfz)
\\    + \int_{\tau'}^{\tau}\int_{\tau'}^{r} \sup_{0\le s<u<v\le t}\frac{|\delta_u\bfz_{vs}^{rr'}|^{2p}}{\psi_{(\alpha,\gamma,\eta,\zeta)}^{1,2}(r,r',v,s)^{2p}|r-r'|^{2}}\,dr^{\prime}\,dr  .
\end{multline}
\end{lemma}
\begin{proof}
First, since $\bfz$ is continuous, we apply Proposition \ref{prop:GRR} to the increment $\bfz_{ts}^\tau-\bfz_{ts}^{\tau'}$, and we get 
\begin{equation}\label{z12 bound}
\frac{\left|\bfz^{\tau\tau^{\prime}}_{ts}\right|}{|\tau^{\prime}-\tau|^{\eta}}\lesssim \left(\int_{\tau^{\prime}}^{\tau}\int_{\tau^{\prime}}^{r}\frac{\left|\bfz^{rr^{\prime}}_{ts}\right|^{2p}}{|r-r^{\prime}|^{2+2p\eta}}dr^{\prime}dr\right)^{1/2p}.
\end{equation}
Moreover, let us write again relation \eqref{relation psi 1 and 12} for the reader's convenience:  
\begin{equation}\label{relation psi 1 and psi 12}
\psi^{1,2}_{\alpha,\gamma,\eta,\zeta}(\tau,\tau^{\prime},t,s)=\left|\tau-\tau^{\prime}\right|^{\eta}\left|\tau^{\prime}-t\right|^{-(\eta-\zeta)}\psi^{1}_{(\alpha,\gamma+\zeta)}(\tau^{\prime},t,s).
\end{equation}
Plugging \eqref{relation psi 1 and psi 12} in \eqref{z12 bound}, we end up with
\begin{equation}\label{228}
\left( \frac{|\bfz^{\tau\tau'}_{ts}|}{\psi^{1,2 }_{(\alpha,\gamma,\eta,\zeta)}(\tau,\tau',t,s)}\right)^{2p}\lesssim I(\tau,\tau',t,s),
 \end{equation}
 where we have set 
\begin{equation*}
I(\tau,\tau',t,s)=\int_{\tau^{\prime}}^{\tau}\int_{\tau^{\prime}}^{r}\frac{\left|\bfz^{rr^{\prime}}_{ts}\right|^{2p}}{|\tau^{\prime}-t|^{-2p(\eta-\zeta)}\left|\psi^{1}_{(\alpha,\gamma+\zeta)}(\tau^{\prime},t,s)\right|^{2p}|r-r^{\prime}|^{2+2p\eta}}dr^{\prime}\,dr.
\end{equation*}
Invoking the fact that $t\le\tau^{\prime}\le r^{\prime}\le \tau$ and  $\eta-\zeta\geq 0$ we have  $|\tau'-t|^{\eta-\zeta}\leq |r'-t|^{\eta-\zeta}$. Hence it immediately follows that  
\begin{equation}\label{upper bound for I}
  I(\tau,\tau',t,s)\lesssim \int_{\tau'}^{\tau}\int_{\tau'}^{r}  \frac{|\bfz^{rr'}_{ts}|^{2p}}{|r'-t|^{-2p(\eta-\zeta)}\left|\psi^1_{(\alpha,\gamma+\zeta)}(r',t,s)\right|^{2p}|r-r'|^{2+2p\eta}}\,dr' \,dr.
\end{equation}
We thus fix $r$ and apply inequality \eqref{U1 bound}  to the Volterra path $(r^{\prime},t,s)\mapsto \bfz^{rr^{\prime}}_{ts}$. We get
\begin{equation}\label{zrr' upper bound}
\left|\bfz^{rr^{\prime}}_{ts}\right|\lesssim \left|r^{\prime}-t\right|^{-(\eta-\zeta)}\psi^{1}_{(\alpha,\gamma+\zeta)}\left(r^{\prime},t,s\right)\left(U^{r^{\prime}}_{(\alpha,\gamma),p,1}(\bfz^{r,\cdot}; \eta,\zeta)+\|\delta\bfz^{r,\cdot}\|^{[s,t]}_{(\alpha,\gamma,\eta,\zeta),1}\|\right).
\end{equation}
We now plug \eqref{zrr' upper bound} into \eqref{upper bound for I}, recall the definition \eqref{U1} of $U^{r^{\prime}}$, resort to~\eqref{relation psi 1 and 12} again and use the expression of \eqref{delta z in [s,t]} for $\|\delta\bfz^{r,\cdot}\|^{[s,t]}_{(\alpha,\gamma,\eta,\zeta),1}$. We end up with 
\begin{equation}\label{I wrt I1 and I2}
I(\tau,\tau^{\prime},t,s)\lesssim I_{1}(\tau,\tau^{\prime})+I_{2}(\tau,\tau^{\prime},t,s),\end{equation}
where $I_{1}$ and $I_{2}$ are respectively given by 
\begin{align*}
     I_{1}(\tau,\tau')=\int_{\tau'}^{\tau}\int_{\tau'}^{r}  \int_0^{r'}\int_0^{v} \frac{|\bfz^{rr'}_{vu}|^{2p}}{ \left|\psi^{1,2}_{(\alpha,\gamma,\eta,\zeta)}(r,r',v,u)\right|^{2p}|v-u|^2|r-r'|^{2}}\,du\,dv\,dr'\,dr,    \\
I_{2}(\tau,\tau',t,s)=\int_{\tau'}^{\tau}\int_{\tau'}^{r} \sup_{0\le s<u<v\le t}\frac{|\delta_u\bfz_{vs}^{rr'}|^{2p}}{\left|\psi_{(\alpha,\gamma+\zeta)}^1(r',v,s)\right|^{2p}|r'-v|^{-2p(\eta-\zeta)}|r-r'|^{2+2p\eta}}\,dr'\,dr.
\end{align*}
Going back to \eqref{U12}, it is now readily checked that 
\begin{equation}\label{I1 upper bound}
     I_{1}(\tau,\tau^{\prime})\le \left(U^{\tau}_{(\alpha,\gamma,\eta,\zeta),p,1,2}\left(\bfz\right)\right)^{2p}\le \left(U^{T}_{(\alpha,\gamma,\eta,\zeta),p,1,2}\left(\bfz\right)\right)^{2p}.
\end{equation}
Furthermore, another application of \eqref{relation psi 1 and 12} reveals that 
\begin{equation}\label{I2 upper bound}
 I_{2}(\tau,\tau',t,s)=\int^{\tau}_{\tau^{\prime}}\int^{r}_{\tau^{\prime}} \sup_{0\le s<u<v\le t}\frac{|\delta_{u} \bfz^{rr'}_{vs}|^{2p}}{ \psi^{1,2}_{(\alpha,\gamma,\eta,\zeta)}(r,r',v,u)^{2p}|r-r'|^{2}}\,dr' \,dr.
\end{equation}
Plugging \eqref{I1 upper bound}-\eqref{I2 upper bound} into \eqref{I wrt I1 and I2} and then back to in \eqref{228}, this achieves the proof of our claim~\eqref{bound}.
\end{proof}

Now we will combine Lemma \ref{GRR lemma volterra} and 
\ref{lem: rec inc bound} to obtain a modified Garsia-Rodemich-Rumsey inequality tailored to Volterra rough paths. 
\begin{thm}\label{Volterra GRR inequality}
Let $\bfz:\Delta_3\rightarrow \RR^d$. For $\alpha,\gamma\in (0,1)$, $\eta, \zeta\in [0,1]$ satisfy the relation \eqref{condition parameters}, we assume that $\delta \bfz \in \cv^{(\alpha,\gamma,\eta,\zeta)}$ where $\cv^{(\alpha,\gamma,\eta,\zeta)}$ is introduced in Definition \ref{def of delta norm}.
  Suppose $\kappa\in (0,\alpha)$. Then for any $p> \frac{1}{\alpha-\kappa} \vee \frac{1}{\zeta}$, the following two bounds holds: 
\begin{align}\label{eq:bound 1}
    \|\bfz\|_{(\kappa,\gamma),1}&\lesssim U^T_{(\kappa,\gamma),1,p}(\bfz) +\|\delta \bfz\|_{(\kappa,\gamma),1}, 
    \\\label{eq:bound 2}
     \|\bfz\|_{(\kappa,\gamma,\eta,\zeta),1,2}&\lesssim U^T_{(\kappa,\gamma,\eta,\zeta),1,2,p}(\bfz) 
     +\|\delta\bfz\|_{\lp\kappa,\gamma,\eta+\frac{1}{p},\zeta+\frac{1}{p}\rp,1,2} \, T^{2+\alpha-\kappa-\frac{1}{p}}.
\end{align}
   
\end{thm}

\begin{proof}
We begin by proving \eqref{eq:bound 1}. It follows directly from \eqref{eq:eta=zeta=0 bound} that for any $0<\kappa<\alpha$
\begin{equation*}
    |\bfz_{ts}^\tau|\lesssim \psi^{1}_{(\kappa,\gamma)}(\tau,t,s)\left(U^{\tau}_{(\kappa,\gamma),p,1}(\bfz)+ \|\delta \bfz\|_{(\alpha,\gamma),1}\right).
    \end{equation*} 
Using that $\tau\mapsto U^\tau$ is increasing (see Remark \ref{rem:increasing}) and taking supremum over $\tau$ on the right hand side above, it is easily seen that \eqref{eq:bound 1} holds. We now move on to prove \eqref{eq:bound 2}. To this aim, we shall spell out the right hand side of \eqref{bound} in a slightly different way. Namely note that for $\delta\bfz\in \cv^{(\alpha,\gamma,\eta,\zeta)}$ and $\eta<\eta^{\prime}$, we have  
\begin{multline}\label{eq:dint delta}
\int_{\tau'}^{\tau}\int_{\tau'}^{r} \sup_{0\le s<u<v\le t}\frac{|\delta_u\bfz_{vs}^{rr'}|^{2p}}{\psi_{(\kappa,\gamma,\eta,\zeta)}^{1,2}(r,r',v,s)^{2p}|r-r'|^{2}}\,dr' \,dr\\
\lesssim 
\|\delta \bfz\|_{\lp \alpha,\gamma,\eta+\frac{1}{p},\zeta+\frac{1}{p}\rp,1,2}^{2p} 
\int_{\tau'}^{\tau}\int_{\tau'}^{r} \sup_{0\le s<u<v\le t}\frac{\psi_{(\alpha,\gamma,\eta+\frac{1}{p},\zeta+\frac{1}{p})}^{1,2}(r,r',v,s)^{2p}}{\psi_{(\kappa,\gamma,\eta,\zeta)}^{1,2}(r,r',v,s)^{2p}|r-r'|^{2}}\,dr'\,dr, 
\end{multline}
where we have used the Definition \ref{V norm b} of the $(1,2)$-norm.
Furthermore, since we have assumed $p>\frac{1}{\alpha-\kappa}$ and $s,v\in [0,T]$ it is readily checked that   
\begin{equation}\label{eq:frac bound}
\frac{\psi_{(\alpha,\gamma,\eta+\frac{1}{p},\zeta+\frac{1}{p})}^{1,2}(r,r',v,s)^{2p}}{\psi_{(\kappa,\gamma,\eta,\zeta)}^{1,2}(r,r',v,s)^{2p}|r-r'|^{2}} 
 \lesssim |v-s|^{2p(\alpha-\kappa)-2}\le T^{2p(\alpha-\kappa)-2}.
\end{equation}
 Hence the right hand side of \eqref{bound} can be upper bounded by 
 \[
 C_{T,p,\alpha,\kappa} \|\delta \bfz\|^{2p}_{(\alpha,\gamma,\eta+\frac{1}{p},\zeta+\frac{1}{p}),1,2}.
 \]
 Plugging this information into \eqref{bound}, the proof of \eqref{eq:bound 2} is now easily achieved.
\end{proof}

\section{Volterra rough path driven by fractional Brownian motion}\label{Volterra rough path driven by fractional Brownian motion}
In this section, we are going to construct the Volterra rough path driven by a fractional Brownian motion with Hurst parameter $H>1/2$. As mentioned in the introduction, this regime leads to nontrivial rough paths development in the Volterra case, due to the singularity of the kernel $k$ in~\eqref{eq:intro Volterra}. Indeed, this singularity pushes down the overall regularity of the Volterra path, so that a singularity of order $\gamma$ yields a regularity  $H-\gamma$ of the volterra path constructed from the fBM (which is thus allowed to be smaller than $\frac{1}{2}$). 

 Let us first recall some basic facts about the stochastic calculus of variations with respect to fractional Brownian motion.

\subsection{Malliavin calculus preliminaries}\label{Malliavin calculus preliminaries}
This section is devoted to review some elementary information on Malliavin calculus (mostly borrowed from \cite{Nualart}) that we will use in Section \ref{first level of the Volterra rough path} and Section \ref{second level of the Volterra rough path}. We first introduce the notation for our main process of interest.  

 \begin{notation}\label{fBm}
In the sequel we denote by $B=\left\{(B^{1}_{t},\ldots,B^{m}_{t}),\,\, t\in [0,T]\right\}$ a standard $m$-dimensional fractional Brownian motion with Hurst parameter $H\in\left(1/2,1\right)$. Recall that $B$ is a centered Gaussian process with independent coordinates. For each component $B^{i}$, the covariance function $R$ is defined by 
\begin{equation}\label{covariance for B}
R(s,t)=\frac{1}{2}\left(|t|^{2H}+|s|^{2H}-|t-s|^{2H}\right).
\end{equation}
\end{notation}
We now say a few words about Cameron-Martin type spaces related to each component $B^{i}$ in Notation \ref{fBm}. Namely let $\mathcal{H}$ be the Hilbert space defined as the closure of the set of step functions on the interval $[0,T]$ with respect to the scalar product
\begin{equation*}
\langle \mathbbm{1}_{[0,t]},\mathbbm{1}_{[0,s]}\rangle_{\mathcal{H}}=\frac{1}{2}\left(t^{2H}+s^{2H}-|t-s|^{2H}\right).
\end{equation*}
Under the assumption $H>1/2$, it is easy to see that the covariance of the fBm \eqref{covariance for B} can be written as 
\begin{equation*}
R(s,t)=a_{H}\int_{0}^{t}\int_{0}^{s}|u-v|^{2H-2}dudv,
\end{equation*}
where the constant $a_{H}$ is defined by $a_{H}=H(2H-1)$. This implies that 
\begin{equation}\label{inner product}
\langle f, g\rangle_{\mathcal{H}}=a_{H}\int_{0}^{T}\int_{0}^{T}f_{u}g_{v}|u-v|^{2H-2}dudv,
\end{equation}
for any pair of step functions $f$ and $g$ on $[0,T]$. Therefore $\mathcal{H}$ can also be seen as the completion of step functions with respect to the inner product \eqref{inner product}. We now introduce a family of additional spaces $\mathcal{|H|}^{\otimes l}$ which will be useful for our computations. Namely for $l\ge 1$ we define $\mathcal{|H|}^{\otimes l}$ as the linear space of measurable functions $f$ on $[0,T]^{l}\subset \mathbb{R}^{l}$ such that
\begin{equation}\label{Hl norm}
\|f\|^{2}_{|\mathcal{H}|^{\otimes l}}:=a^{l}_{H}\int_{[0,T]^{2l}}|f_{\bfu}||f_{\bfv}||u_{1}-v_{1}|^{2H-2}\cdots|u_{l}-v_{l}|^{2H-2}d\bfu d\bfv<\infty,
\end{equation}
where we write $\bfu=(u_{1},\cdots, u_{l}),\, \, \bfv=(v_{1},\ldots, v_{l})\in [0,T]^{l}$. Notice that $\mathcal{|H|}^{\otimes l}$ is a subset of $\mathcal{H}^{\otimes l}$. The main interest of the spaces $|\ch|$ is due to the fact that while $\mathcal{H}^{\otimes l}$ contains distributions, the space $\mathcal{|H|}^{\otimes l}$ is a space of functions. 

For each component $B^{i}$, the mapping $\mathbbm{1}_{[0,t]} \mapsto B^{i}_{t}$ can be extended to a linear isometry between $\mathcal{H}$ and the Gaussian space spanned by $B^{i}$. We denote this isometry by $h\mapsto B^{i}(h)$. In this way, $\{B^{i}(h), h\in \mathcal{H}\} $ is an isonormal Gaussian process indexed by the Hilbert space $\mathcal{H}$. Namely, we have
\begin{equation}\label{isometry}
\mathbb{E}\left[B^{i}(f)\,B^{i}(g)\right]=\langle f,g\rangle_{\mathcal{H}}.
\end{equation}
It is also worth mentioning that the Wiener integral can be approximated by Riemann type sums. Namely for $h\in\mathcal{H}$ the following limit holds true in $L^{2}(\Omega)$:
\begin{equation}\label{def of wiener}
B^{i}(h)=\lim_{|\mathcal{P}|\to 0}\sum_{[r,v]\in\mathcal{P}}B^{i}_{vr} \, h({r}),
\end{equation}
where the Riemann sum is written similarly to \eqref{convolution in 1d} and we recall that $B^{i}_{vr}=B^{i}_{v}-B^{i}_{r}$.

Let $\mathcal{S}$ be the set of smooth and cylindrical random variables of the form 
\[
F=f(B_{s_{1}},\ldots,B_{s_{N}}),
\]
where $N\ge 1$ and $f\in C_{b}^{\infty}(\mathbb{R}^{m\times N})$. For each $j=1,\ldots, m$ and $t\in [0,T]$, the partial Malliavin derivative of $F$ with respect to the component $B^{j}$ is defined for $F\in \mathcal{S}$ as the $\mathcal{H}$-valued random variable 
\begin{equation}\label{derivative}
D^{j}_{t}F=\sum_{i}^{N}\frac{\partial f}{\partial x^{j}_{i}}(B_{s_{1}},\ldots,B_{s_{N}})\mathbbm{1}_{[0,s_{i}]}(t),\qquad t\in[0,T],
\end{equation}
where $x^{j}_{i}$ stands for the $j$-th component of $x$. We can iterate this procedure to define higher order derivatives $D^{j_{1},\ldots, j_{l}}F$, which take values in $\mathcal{H}^{\otimes l}$. For any $p\ge 1$ and integer $k\ge 1$, we define the Sobolev space $\mathbb{D}^{k,p}$ as the closure of $\mathcal{S}$ with respect to the norm 
\begin{equation}\label{Sobolev norm}
\|F\|_{k,p}^{p}=\mathbb{E}[|F|^{p}]+\mathbb{E}\left[\sum_{i=1}^{k}\left(\sum_{j_{1},\ldots,j_{l}}^{m}\|D^{j_{i},\ldots,j_{l}}F\|^{2}_{\mathcal{H}^{\otimes l}}\right)^{p/2}\right].
\end{equation} 
If $V$ is Hilbert space, $\mathbb{D}^{k,p}(V)$ denotes the corresponding Sobolev space of $V$-valued random variables.

For any $j=1,\ldots,m$, we denote by $\delta^{\diamond,j}$ the adjoint of the derivative operator $D^{j}$. For a process $\{u_{t};\, t\in[0,T]\}$, we say $u\in \mathrm{Dom}\, \delta^{\diamond,j}$ if there is a $\delta^{\diamond,j}(u)\in L^{2}(\mathbb{R}^{m})$ such that for any $F\in \mathbb{D}^{k,p}$ the following duality relation holds 
\begin{equation}\label{48}
\mathbb{E}\left[\langle u,D^{j}F \rangle_{\mathcal{H}}\right]=\mathbb{E}\left[\delta^{\diamond,j}(u)F\right].
\end{equation}
The random variable $\delta^{\diamond,j}(u)$ is also called the Skorohod integral of $u$ with respect to the fBm $B^{j}$, and we use the notation $\delta^{\diamond,j}(u)=\int_{0}^{T}u_{t}\delta^{\diamond} B^{j}_{t}$. It is well known that $\mathbb{D}^{1,2}(\mathcal{H})\subset \mathrm{Dom}\,(\delta^{\diamond,j})$ for all $j=1,\ldots,m$.

We now introduce a pathwise type integral defined on the Wiener space, called Stratonovich integral. Namely let $u=\{u_{t},t\in [0,T]\}$ be a continuous stochastic process, and let $\mathcal{P}$ be a generic partition of $[s,t]$. Following \cite[Section 3.1]{Nualart}, we define
\begin{align}\label{def of S}
B^{i,\mathcal{P}}_{t}=\sum_{[r,v]\in\mathcal{P}}\frac{B^{i}_{vr}}{v-r}\mathbbm{1}_{[r,v]}(t),
\qquad \text{and} \qquad
S^{i,\mathcal{P}}_{ts}=\int_{s}^{t}u_{r} \, B^{i,\mathcal{P}}_{r}dr.
\end{align}
Then the Stratonovich integral of $u$ with respect to $B^{i}$ is defined as \begin{equation}\label{def of stra}
\int_{s}^{t}u_{r} \, dB^{i}_{r}=\lim_{|\mathcal{P}|\to 0}S^{i,\mathcal{P}}_{ts},
\end{equation}
where the limit is understood in probability. On the other hand, assume that $u$ is $C^{\kappa}$-H\"older with $\kappa+H>1$. Moreover we suppose that $u\in \mathbb{D}^{1,2}(\mathcal{H})$ and the derivative $D^{j}_{s}u_{t}$ exists and satisfies 
\[
\int_{0}^{T}\int_{0}^{T} |D^{j}_{s}u_{t}||t-s|^{2H-2}ds\,dt<\infty \quad \textnormal{a.s}, 
\qquad \text{and} \qquad 
\mathbb{E}\left[\|D^{j}u\|^{2}_{\mathcal{|H|}^{\otimes 2}}\right]<\infty.
\]
Then the Stratonovich integral $\int_{0}^{T}u_{t} d B^{j}_{t}$ exists, and we have the following relation between Skorohod and Stratonovich stochastic integrals:
\begin{equation}\label{relation integrals}
\int_{0}^{T} u_{t} d B^{j}_{t}=\int_{0}^{T}u_{t}\delta^{\diamond} B^{j}_{t}+a_{H}\int_{0}^{T}\int_{0}^{T}D^{j}_{s}\,u_{t}|t-s|^{2H-2}ds\,dt.
\end{equation}
We close this section by spelling out  Meyer's inequality (see \cite[Proposition 1.5.4]{Nualart}) for the Skorohod integral: given $p>1$ and an integer $k\ge 1$, there is a constant $c_{k,p}$ such that the $k$-th iterated Skorohod integral satisfies 
\begin{equation}\label{delta u bound}
\|(\delta^{\diamond})^{k}(u)\|_{p}\le c_{k,p}\|u\|_{\mathbb{D}^{k,p}(\mathcal{H}^{\otimes k})},
\qquad \text{for all} \quad u\in {\mathbb{D}^{k,p}(\mathcal{H}^{\otimes k})}.
\end{equation}

\subsection{First level of the Volterra rough path}\label{first level of the Volterra rough path}
In this section, we will construct the first level of the Volterra rough path driven by a fBm as introduced in Notation \ref{fBm}. We start by defining our main object of study. 
\begin{Def}\label{def of z1}
Consider a fractional Brownian motion $B :[0,T]\rightarrow \mathbb{R}^m$ with Hurst parameter $H$ as given in Notation~\ref{fBm}, and a function $h$ of the form $h^{\tau}_{ts}(r)=(\tau-r)^{-\gamma}\mathbbm{1}_{[s,t]}(r)$. We assume that $H$, $\gamma$ satisfy
$H\in\left(1/2, 1\right)$,  $\gamma\in~(0, 2H-1)$.
Then for $(s,t,\tau)\in \Delta_{3}$ we define the increment $\bfz^{1,\tau,i}_{ts}=\int_{s}^{t} \left(\tau-r\right)^{-\gamma}dB^{i}_{r}$ as a Wiener integral of the form 
\begin{equation}\label{z1 expression fbm}
\bfz_{ts}^{1,\tau,i}:=B^{i}(h^{\tau}_{ts}).
\end{equation}
\end{Def}
\begin{remark}
Note that for the particular type of integrand $h$ considered in Definition \ref{def of z1}, the process $B^i(h^\tau_{ts})$ is additive in its lower variables, in the sense that  
\begin{equation}
B^i(h^\tau_{ts})=B^i(h^\tau_{t0})-B^i(h^\tau_{s0}).
\end{equation}
Thus defining $\bfz^{1,\tau}_t:= \bfz^{1,\tau}_{t0}$ we have that $\bfz^1$ is defined on the simplex $\Delta_2$. 
\end{remark}

With Definition \ref{def of z1} in hand, we now estimate the second moment of $\bfz^{1,\tau,i}_{ts}$ and $\bfz^{1,\tau\tau^{\prime},i}_{ts}$. 
\begin{lemma}\label{z1 second moment}
Consider the Volterra rough path $\bfz^{1}$ as given in \eqref{z1 expression fbm}, and four parameters $H\in~\left(1/2,1\right)$, $\gamma\in (0,1)$, $\eta,\, \zeta\in[0,1]$ satisfying 
\begin{eqnarray}\label{H, gamma, eta, zeta conditions}
 \gamma<2H-1, \qquad \text{and} \qquad 0\le \zeta \le \inf(H-\gamma, \eta).
\end{eqnarray}
Then for $(s,t,\tau)\in \Delta_{3}$, we have
\begin{equation}\label{2nd moment of z1}
\mathbb{E}[(\bfz^{1,\tau,i}_{ts})^{2}]\lesssim \left|\psi^{1}_{(H,\gamma)}\left(\tau,t,s\right)\right|^{2}.
\end{equation}
In addition for $(s,t,\tau^{\prime},\tau)\in\Delta_{4}$, we get
\begin{equation}\label{2nd moment of z1 2}
\mathbb{E}[(\bfz^{1,\tau\tau^{\prime},i}_{ts})^{2}]\lesssim\left|\psi^{1,2}_{(H,\gamma,\eta,\zeta)}\left(\tau,,\tau^{\prime},t,s\right)\right|^{2},
\end{equation}
where $\psi^{1}$ and $\psi^{1,2}$ are given in Notation \ref{notation of psi}.
\end{lemma}
\begin{proof}
We first prove relation \eqref{2nd moment of z1}. According to \eqref{z1 expression fbm} and \eqref{isometry}, we can compute $\mathbb{E}[(\bfz^{1,\tau,i}_{ts})^{2}]$ as 
\begin{equation}\label{expression of 2nd z1}
\mathbb{E}[(\bfz^{1,\tau,i}_{ts})^{2}]=\mathbb{E}\left[B^{i}(h^{\tau}_{ts})\,B^{i}(h^{\tau}_{ts})\right]=\left\langle h^{\tau}_{ts}, h^{\tau}_{ts}\right\rangle_{\mathcal{H}}
\end{equation}
Owing to relation \eqref{inner product} for the inner product in $\mathcal{H}$, we thus obtain
\begin{equation}\label{expression of 2nd z1}
\mathbb{E}[(\bfz^{1,\tau,i}_{ts})^{2}]=H(2H-1)\iint_{[s,t]\times[s,t]}(\tau-r)^{-\gamma}(\tau-l)^{-\gamma}\left|r-l\right|^{2H-2}drdl. 
\end{equation}
Notice that the function $(\tau-r)^{-\gamma}(\tau-l)^{-\gamma}\left|r-l\right|^{2H-2}$ is symmetric. Hence we can recast \eqref{expression of 2nd z1} as
\begin{equation}\label{expression of 2nd z1 1}
\mathbb{E}[(\bfz^{1,\tau,i}_{ts})^{2}]=2H(2H-1)\int_{s}^{t}(\tau-r)^{-\gamma}dr\int_{r}^{t}(\tau-l)^{-\gamma}(l-r)^{2H-2}dl. 
\end{equation}
In the right hand side of \eqref{expression of 2nd z1 1}, we first estimate the integral 
\begin{equation}\label{expression J}
\int_{r}^{t}(\tau-l)^{-\gamma}(l-r)^{2H-2}dl:=J.
\end{equation}
Since $l\in (r,t)$ in \eqref{expression J}, we proceed to a change of variable $l=r+\theta(t-r)$. We obtain 
\begin{equation}\label{expression J 2}
J=(t-r)^{2H-1}\int_{0}^{1}\left(\tau-r-\theta(t-r)\right)^{-\gamma}\theta^{2H-2}d\theta
\le (t-r)^{2H-1}(\tau-r)^{-\gamma} \int_{0}^{1}(1-\theta)^{-\gamma}\theta^{2H-2}d\theta.
\end{equation}
Recall that we have assumed that $\gamma <2H-1<1$. Moreover $H>1/2$ and thus $2H-2>-1$. Hence the right hand side of \eqref{expression J 2} can be expressed in terms of Beta functions in the following way:
\[
 \int_{0}^{1}(1-\theta)^{-\gamma}\theta^{2H-2}d\theta=\mathrm{Beta}(1-\gamma,2H-1)< \infty.
 \]
Reporting this identity in the right hand side of \eqref{expression J 2}, we end up with
\begin{equation}\label{J bound}
J\lesssim (t-r)^{2H-1}(\tau-r)^{-\gamma}.
\end{equation}
Plugging \eqref{J bound} into \eqref{expression of 2nd z1}, we thus get 
\begin{equation}\label{318}
\mathbb{E}[(\bfz^{1,\tau,i}_{ts})^{2}]\lesssim \int_{s}^{t}(\tau-r)^{-2\gamma}(t-r)^{2H-1}dr.
\end{equation}
We now bound the right hand side of \eqref{318} in two different ways. First since $(\tau-r)>(t-r)$, we have
\begin{equation}\label{bound 1}
\mathbb{E}[(\bfz^{1,\tau,i}_{ts})^{2}]\lesssim \int_{s}^{t}(t-r)^{2H-2\gamma-1}dr\lesssim (t-s)^{2H-2\gamma},
\end{equation}
where we have resorted to the fact that $\gamma<2H-1<H$ for the second inequality. Next we also use the fact that $(\tau-r)>(\tau-t)$ in the right hand side of \eqref{318}, which allows to write
\begin{equation}\label{bound 2}
\mathbb{E}[(\bfz^{1,\tau,i}_{ts})^{2}]\lesssim (\tau-t)^{-2\gamma}\int_{s}^{t}(t-r)^{2H-1}dr\lesssim (\tau-t)^{-2\gamma}(t-s)^{2H}.
\end{equation}
Combining \eqref{bound 1} and \eqref{bound 2}, we end up with the following estimate for the second moment of $\bfz^{1,\tau,i}_{ts}$:
\begin{equation}\label{desired result}
\mathbb{E}[(\bfz^{1,\tau,i}_{ts})^{2}]\lesssim \left[(\tau-t)^{-2\gamma}(t-s)^{2H}\right]\wedge (t-s)^{2H-2\gamma}=\left(\psi^{1}_{(H,\gamma)}(\tau,t,s)\right)^{2},
\end{equation}
where we have appealed to the definition \eqref{psi 1} of $\psi^{1}$ for the second identity.
Relation \eqref{desired result} is the desired result \eqref{2nd moment of z1}.

Next, we will prove inequality \eqref{2nd moment of z1 2}. To this aim, we first note that owing to \eqref{z1 expression fbm}, we have the following expression for $\bfz^{1,\tau\tau^{\prime},i}_{ts}$,
\begin{equation}\label{expression of z12}
\bfz^{1,\tau\tau^{\prime},i}_{ts}=\bfz^{1,\tau,i}_{ts}-\bfz^{1,\tau^{\prime},i}_{ts}=B^{i}(h^{\tau}_{ts}-h^{\tau^{\prime}}_{ts}).
\end{equation}
Similarly to \eqref{expression of 2nd z1 1}, we can thus rewrite $\mathbb{E}[(\bfz^{1,\tau\tau^{\prime},i}_{ts})^{2}]$ as
\begin{multline}\label{expression of 2nd z1 2}
\mathbb{E}[(\bfz^{1,\tau\tau^{\prime},i}_{ts})^{2}]=\langle h^{\tau}_{ts}-h^{\tau^{\prime}}_{ts}, h^{\tau}_{ts}-h^{\tau^{\prime}}_{ts}\rangle_{\mathcal{H}}\\
=2H(2H-1) \int_{s}^{t}\left[(\tau^{\prime}-r)^{-\gamma}-(\tau-r)^{-\gamma}\right]dr\,\int_{r}^{t}\left[(\tau^{\prime}-l)^{-\gamma}-(\tau-l)^{-\gamma}\right](l-r)^{2H-2}dl.
\end{multline}
We now recall an elementary inequality on increments of negative power functions. Namely for $\tau>\tau^{\prime}>r$ and $\eta\in [0,1]$ we have 
 \[
 (\tau^{\prime}-r)^{-\gamma}-(\tau-r)^{-\gamma}\lesssim (\tau-\tau^{\prime})^{\eta}(\tau^{\prime}-r)^{-\eta-\gamma} .
 \]
  Plugging this upper bound into the right hand side of \eqref{expression of 2nd z1 2}, we obtain
 \begin{equation}\label{324}
\mathbb{E}[(\bfz^{1,\tau\tau^{\prime},i}_{ts})^{2}]\lesssim|\tau-\tau^{\prime}|^{2\eta}\int_{s}^{t}(\tau^{\prime}-r)^{-\eta-\gamma}\int_{r}^{t}(\tau^{\prime}-l)^{-\eta-\gamma}(l-r)^{2H-2}dl.
\end{equation}
The expression \eqref{324} is now very similar to \eqref{expression of 2nd z1 1}. Therefore with the same steps as for \eqref{expression J}-\eqref{bound 2}, for some $\zeta\in [0,H-\gamma)$ and $\eta \in [\zeta,1]$ we get
\begin{equation}\label{desired result 2}
\mathbb{E}[(\bfz^{1,\tau\tau^{\prime},i}_{ts})^{2}]\lesssim |\tau-\tau^{\prime}|^{2\eta}|\tau^{\prime}-t|^{-2(\eta-\zeta)}\left(\left[(\tau-t)^{-2\gamma-2\zeta}(t-s)^{2H}\right]\wedge (t-s)^{2H-2\gamma-2\zeta}\right).
\end{equation}
According to the definition \eqref{psi 12} of $\psi^{1,2}$,  \eqref{desired result 2} is equivalent to 
 \begin{equation*}
 \mathbb{E}[(\bfz^{1,\tau\tau^{\prime},i}_{ts})^{2}]\lesssim \left|\psi^{1,2}_{(H,\gamma,\eta,\zeta)}(\tau,\tau^{\prime},t,s)\right|^{2}.
 \end{equation*}
This finishes the proof of \eqref{2nd moment of z1 2}. 
\end{proof}
\begin{remark}
One can easily extend the computation of Lemma \ref{2nd moment of z1} in order to get more general bounds for covariance functions. Namely for any $(s,u,v,\tau)\in \Delta_{4}$, and recalling the expression~\eqref{psi 1} for $\psi^{1}$ we have
\begin{equation}\label{ZusZvs}
\left|\mathbb{E}[\bfz^{1,\tau,i}_{us}\bfz^{1,\tau,i}_{vs}]\right|\lesssim\left|\psi^{1}_{(H,\gamma)}(\tau,v,s)\right|^{2}.
\end{equation}
Similarly for any $(s,u,v,\tau^{\prime},\tau)\in \Delta_{5}$ and recalling our definition \eqref{psi 12} for $\psi^{1,2}$, we obtain
\begin{equation}\label{ZusZvs 2}
\left|\mathbb{E}[\bfz^{1,\tau\tau^{\prime},i}_{us}\bfz^{1,\tau\tau^{\prime},i}_{vs}]\right|\lesssim \left|\psi^{1,2}_{(H,\gamma,\eta,\zeta)}(\tau,\tau^{\prime},v,s)\right|^{2},
\end{equation}
where $H\in~(\frac{1}{2},1)$, $\gamma\in (0,1)$, and $\eta, \zeta\in[0,1]$ satisfy 
  relation \eqref{H, gamma, eta, zeta conditions}.
\end{remark}

\subsection{Second level of the Volterra rough path}\label{second level of the Volterra rough path}In this section we turn our attention to the construction of a nontrivial Volterra rough path above a fBm. More specifically our aim is to construct a family $\{\bfz^{1,\tau},\bfz^{2,\tau}\}$ verifying Definition \ref{hyp 2}. Let us start with the definition of $\bfz^{2,\tau}$. 
\begin{Def}\label{the second level fBm Volterra rough path} We consider a fractional Brownian motion $B :[0,T]\rightarrow \mathbb{R}^m$ as given in Notation~\ref{fBm}, as well as the first level of the Volterra rough path $\bfz^{1,\tau}$ defined by \eqref{z1 expression fbm}. As in Definition~\ref{def of z1}, we assume that $H$,  $\gamma$ satisfy $H\in (\frac{1}{2}, 1)$ and $\gamma\in(0, 2H-1)$. Then for $(s,r,t,\tau)\in \Delta_{4}$, we set
\begin{equation}\label{def of u}
 u^{\tau,i}_{ts}(r)=(\tau-r)^{-\gamma}\bfz^{1,r,i}_{rs}\mathbbm{1}_{[s,t]}(r).
 \end{equation}
 The increment $\bfz^{2,\tau}_{ts}$ is given as follows: if $i\ne j$ we define $\bfz^{2,\tau,i,j}_{ts}$ as
 \begin{equation}\label{z2 ij}
 \bfz^{2,\tau,i,j}_{ts}=B^{j}(u^{\tau,i}_{ts}),
 \end{equation}
where (conditionally on $B^{i}$) the random variable $B^{j}(u^{\tau,i}_{ts})$ has to be interpreted as a Wiener integral.
In the case $i=j$, we set
\begin{equation}\label{z2 ii}
\bfz^{2,\tau,i,i}_{ts}=\int_{s}^{t} u^{\tau,i}_{ts}(r) dB^{i}_{r}, 
\end{equation}
where the right hand side of \eqref{z2 ii} is defined as a Stratonovich integral like \eqref{relation integrals}.
 \end{Def}
\begin{remark}\label{relation zij and zii}
With  Definition \ref{def of z1} of $\bfz^{1,\tau}$ in mind when considering the process $u^{\tau,i}$ in \eqref{def of u}, we get that $\bfz^{2,\tau}$ in \eqref{z2 ij}-\eqref{z2 ii} is formally interpreted as 
\begin{equation}\label{general z2}
\bfz_{ts}^{2,\tau,i,j}=\int_{s}^{t}\left(\tau-r\right)^{-\gamma}\int_{s}^{r}\left(r-l\right)^{-\gamma}dB_{l}^{i} dB_{r}^{j}.
\end{equation}
Below we will show that $\bfz^{2,\tau}$ can indeed be considered as the double iterated integral  in \eqref{general z2}.
\end{remark}
Similarly to what we did for $\bfz^{1}$, we will now estimate the second moment of $\bfz^{2,\tau}$.
\begin{prop}\label{second moment of z2}
Consider the second level $\bfz^{2,\tau}$ of the Volterra rough path, as defined in 
\eqref{z2 ij}-\eqref{z2 ii}. Recall that $H\in(\frac{1}{2},1)$, $\gamma\in (0,1)$, and $\eta,\,\zeta\in[0,1]$ satisfy the relation \eqref{H, gamma, eta, zeta conditions}. Then for $(s,t,\tau)\in \Delta_{3}$ and any $i,j=1,\ldots,d$, we have
\begin{equation}\label{prob reg}
\mathbb{E}\left[\left(\bfz^{2,\tau,i,j}_{ts}\right)^{2}\right]\lesssim \left|\psi^{1}_{(2H-\gamma,\gamma)}(\tau,t,s)\right|^{2}.
\end{equation}
As far as the $(1,2)$-type increments are considered, we get
\begin{equation}\label{prob reg 2}
\mathbb{E}\left[\left(\bfz^{2,\tau\tau^{\prime},i,j}_{ts}\right)^{2}\right]\lesssim \left|\psi^{1,2}_{(2H-\gamma,\gamma,\eta,\zeta)}(\tau,\tau^{\prime},t,s)\right|^{2},
\end{equation}
where $\psi^{1}$ and $\psi^{1,2}$ are given in Notation \ref{notation of psi}. 
\end{prop} 
\begin{proof}
We will prove relation \eqref{prob reg} in the following, \eqref{prob reg 2} can be treated in a similar way and is left to the reader for sake of conciseness. According to Remark \ref{relation zij and zii}, we consider $\bfz^{2,\tau,i,j}_{ts}$ and $\bfz^{2,\tau,i,i}_{ts}$ as different integrals. Therefore we will split the proof of \eqref{prob reg} into two parts: $i\neq j$ and $i=j$. 

\smallskip

\noindent
\emph{Step 1: Relation \eqref{prob reg} for $i\neq j$.} In this step, we will show that \eqref{prob reg} holds for $\bfz^{2,\tau,i,j}_{ts}$ as given in \eqref{z2 ij}. According to Definition \ref{the second level fBm Volterra rough path}, we consider the integral \eqref{z2 ij} as a conditional Wiener integral. Namely due to the independence of $B^{i}$ and $B^{j}$ we can write
\begin{equation}\label{E z2}
\mathbb{E}\left[\left(\bfz_{ts}^{2,\tau,i,j}\right)^{2}\right]=\mathbb{E}\left\{\mathbb{E}\left[\left(\bfz_{ts}^{2,\tau,i,j}\right)^{2}\,\Big|\,B^{i}\right]\right\}=\mathbb{E}\left\{\mathbb{E}\left[\left(B^{j}(u^{\tau,i}_{ts})\right)^{2}\,\Big|\,B^{i}\right]\right\},
\end{equation}
where we recall that $u^{\tau,i}_{ts}$ is defined by \eqref{def of u}. Furthermore, relation \eqref{E z2} for Wiener integral reads
\begin{equation}\label{z2-ij upper bound}
\mathbb{E}\left[\left(B^{j}(u^{\tau,i}_{ts})\right)^{2}\,\Big|\,B^{i}\right]=\|u^{\tau,i}_{ts}\|^{2}_{\mathcal{H}},
\end{equation}
and thus 
\begin{equation}\label{z2 relation}
\mathbb{E}\left[\left(\bfz_{ts}^{2,\tau,i,j}\right)^{2}\right]=\mathbb{E}\left[\|u^{\tau,i}_{ts}\|^{2}_{\mathcal{H}}\right].\end{equation}
In order to bound the right hand side of \eqref{z2 relation}, we resort to the expression \eqref{inner product} for the inner product in $\mathcal{H}$. This yields 
\begin{multline*}
\mathbb{E}\left[\| u^{\tau,i}_{ts}\|^{2}_{\mathcal{H}}\right]=\EE[\langle u^{\tau,i}_{ts}, u^{\tau,i}_{ts}\rangle _{\mathcal{H}}]=H(2H-1)\mathbb{E}\Big[\int_{s}^{t}\int_{s}^{t}\left((\tau-r_{1})^{-\gamma}\int_{s}^{r_{1}}(r_{1}-l_{1})^{-\gamma}dB^{i}_{l_{1}}\right) 
\\
\times \left((\tau-r_{2})^{-\gamma}\int_{s}^{r_{2}}(r_{2}-l_{2})^{-\gamma}dB^{i}_{l_{2}}\right)|r_{1}-r_{2}|^{2H-2}dr_{1}dr_{2}\Big].
\end{multline*}
Thanks to an easy application of Fubini's theorem, and invoking the symmetry of the integrand like in \eqref{expression of 2nd z1 1} we get 
\begin{multline}\label{expression of I1}
\mathbb{E}\left[\| u^{\tau,i}_{ts}\|^{2}_{\mathcal{H}}\right]=2H(2H-1)\int_{s}^{t}\int_{s}^{r_{1}}(\tau-r_{1})^{-\gamma}(\tau-r_{2})^{-\gamma}|r_{1}-r_{2}|^{2H-2}
\\
\times \mathbb{E}\left[\int_{s}^{r_{1}}(r_{1}-l_{1})^{-\gamma}dB^{i}_{l_{1}}\int_{s}^{r_{2}}(r_{2}-l_{2})^{-\gamma}dB^{i}_{l_{2}}\right]dr_{2}dr_{1}.
\end{multline}
Moreover, owing to \eqref{ZusZvs}, and recalling the definition \eqref{psi 1} of $\psi^{1}$, for $r_{2}<r_{1}$ we have
\begin{equation}\label{Zr1sZr2s}
\mathbb{E}\left[\int_{s}^{r_{1}}(r_{1}-l_{1})^{-\gamma}dB^{i}_{l_{1}}\int_{s}^{r_{2}}(r_{2}-l_{2})^{-\gamma}dB^{i}_{l_{2}}\right]=\mathbb{E}\left[\bfz^{1,r_{1},i}_{r_{1}s}\bfz^{1,r_{2},i}_{r_{2}s}\right] \lesssim |r_{1}-s|^{2H-2\gamma}.
\end{equation}
Plugging \eqref{Zr1sZr2s} into \eqref{expression of I1}, we thus get
\begin{equation}\label{I1 upper bound 1}
\mathbb{E}\left[\| u^{\tau,i}_{ts}\|^{2}_{\mathcal{H}}\right]\lesssim \int_{s}^{t}\int_{s}^{r_{1}}(\tau-r_{1})^{-\gamma}(\tau-r_{2})^{-\gamma}|r_{1}-r_{2}|^{2H-2}|r_{1}-s|^{2H-2\gamma}dr_{1}dr_{2}.
\end{equation}
Similarly to what we did for \eqref{expression of 2nd z1}-\eqref{bound 2} in the proof of Lemma \ref{z1 second moment}, we evaluate the right hand side of \eqref{I1 upper bound 1} thanks to elementary integral bounds and the  use of $\mathrm{Beta}$ functions. We let the patient reader check that we get \begin{equation}\label{I1 bound}
\mathbb{E}\left[\| u^{\tau,i}_{ts}\|^{2}_{\mathcal{H}}\right]\lesssim \left[|\tau-t|^{-2\gamma}|t-s|^{4H-2\gamma}\right]\wedge |t-s|^{4H-4\gamma}.
\end{equation}
Plugging \eqref{I1 bound} into \eqref{z2-ij upper bound}, we thus obtain
\begin{equation}\label{skorohod upper bound}
\mathbb{E}\left[\left(\bfz_{ts}^{2,\tau,i,j}\right)^{2}\right]\lesssim \left[|\tau-t|^{-2\gamma}|t-s|^{4H-2\gamma}\right]\wedge |t-s|^{4H-4\gamma}=\left|\psi^{1}_{(2H-\gamma,\gamma)}(\tau,t,s)\right|^{2},
\end{equation}
where we have invoked the definition \eqref{psi 1} of $\psi^{1}$. This is the desired result \eqref{prob reg}.

\smallskip

\noindent
\emph{Step 2: Relation \eqref{prob reg} for $i=j$.} In this step, we will show that relation \eqref{prob reg} holds for $\bfz^{2,\tau,i,i}_{ts}$ defined by \eqref{z2 ii}. According to Definition \ref{the second level fBm Volterra rough path}, we consider \eqref{z2 ii} as a Stratonovich integral like~\eqref{relation integrals}. We thus recast   \eqref{z2 ii} as
\begin{equation}\label{relation integrals 2}
\bfz^{2,\tau,i,i}_{ts}=\int_{s}^{t} u^{\tau,i}_{ts}(r) d B^{i}_{r}
=\int_{s}^{t}u^{\tau,i}_{ts}(r)\, \delta^{\diamond}\! B^{i}_{r}
+H(2H-1)\int_{s}^{t}\int_{s}^{t}D^{i}_{l}\left(u^{\tau,i}_{ts}(r)\right)|r-l|^{2H-2}dr\,dl,
\end{equation}
where $u^{\tau,i}_{ts}(r)$ as given in \eqref{def of u}. Taking square and expectation on both sides of \eqref{relation integrals 2}, we obtain
\begin{equation}\label{second moment of stratonovich}
\mathbb{E}\left[\left(\int_{s}^{t} u^{\tau,i}_{ts}(r) d B^{i}_{r}\right)^{2}\right]\lesssim J_{1}+J_{2},
\end{equation}
where the terms $J_{1}$ and $J_{2}$ are respectively defined by 
\begin{align}
&J_{1}=\mathbb{E}\left[\left(\int_{s}^{t}u^{\tau,i}_{ts}(r) \, \delta^{\diamond}\!  B^{i}_{r}\right)^{2}\right]\label{def J1},
\\
&J_{2}=\left(H(2H-1)\right)^{2}\mathbb{E}\left[\left(\int_{s}^{t}\int_{s}^{t}D^{i}_{l}\left(u^{\tau,i}_{ts}(r)\right)|r-l|^{2H-2}dr\,dl\right)^{2}\right]\label{def J2}.
\end{align}
In the following, we will estimate $J_{1}$ and $J_{2}$ separately. 

In order to upper bound $J_{1}$, we recall that the integral $\int_{s}^{t}u^{\tau,i}_{ts}(r)\delta^\diamond B^{i}_{r}$ in the right hand side of~\eqref{def J1} is interpreted as a Skorohod integral of the form $\delta^\diamond (u^{\tau,i}_{ts})$. Resorting to \eqref{delta u bound}, we thus have
\begin{equation}\label{delta u bound 1}
J_{1}=\|\delta^\diamond \left(u^{\tau,i}_{ts}\right)\|^{2}_{2}\lesssim \|u^{\tau,i}_{ts}\|^{2}_{\mathbb{D}^{1,2}(\mathcal{H}^{\otimes 2})}.
\end{equation}
Let us now handle the right hand side of \eqref{delta u bound 1}. Owing to \eqref{Sobolev norm}, we get 
\begin{equation}\label{upper bound J1 1}
J_{1}\lesssim \mathbb{E}\left[\| u^{\tau,i}_{ts}\|^{2}_{\mathcal{H}}\right]+\mathbb{E}\left[\|D^{i}\left(u^{\tau,i}_{ts}\right)\|^{2}_{\mathcal{H}^{\otimes 2}}\right].
\end{equation}
Notice that the first term of the right hand side of \eqref{upper bound J1 1} is what we upper bounded in Step 1. Thanks to \eqref{I1 bound}, we obtain
\begin{equation}\label{J1 bound 1}
\mathbb{E}\left[\| u^{\tau,i}_{ts}\|^{2}_{\mathcal{H}}\right]\lesssim \left[|\tau-t|^{-2\gamma}|t-s|^{4H-2\gamma}\right]\wedge |t-s|^{4H-4\gamma}.
\end{equation}
In order to estimate the second term in the right hand side of \eqref{upper bound J1 1}, let us first compute the partial Malliavin derivative $D^{i}_{l}(u^{\tau,i}_{ts}(r))$ of $u^{\tau,i}_{ts}(r)$ with respect to $B^{i}$. Specifically, we gather \eqref{z1 expression fbm} and \eqref{def of u} in order to get  
\begin{equation*}
u^{\tau,i}_{ts}(r)=(\tau-r)^{-\gamma}B^{i}(h^{r}_{rs})\mathbbm{1}_{[s,t]}(r),\qquad \text{with}\qquad h^{r}_{rs}(l)=(r-l)^{-\gamma}\mathbbm{1}_{[s,r]}(l).\end{equation*}
Thanks to \eqref{derivative}, we thus get
\begin{equation}\label{Du 1}
D^{i}_{l}\left(u^{\tau,i}_{ts}(r)\right)=(\tau-r)^{-\gamma}h^{r}_{rs}(l)\mathbbm{1}_{[s,t]}(r)=(\tau-r)^{-\gamma}(r-l)^{-\gamma}\mathbbm{1}_{[s,r]}(l)\mathbbm{1}_{[s,t]}(r).
\end{equation}
Plugging \eqref{Du 1} into the second term of the right hand side of \eqref{upper bound J1 1}, and having the definition~\eqref{Hl norm} of $\mathcal{H}^{\otimes 2}$-norms in mind, we obtain
\begin{multline}\label{Du bound}
\|D^{i}u^{\tau,i}_{ts}\|^{2}_{\mathcal{H}^{\otimes 2}}=\left(H(2H-1)\right)^{2}\int_{s}^{t}\int_{s}^{t}\int_{s}^{r_{1}}\int_{s}^{r_{2}}(\tau-r_{1})^{-\gamma}(\tau-r_{2})^{-\gamma}\\
\times(r_{1}-l_{1})^{-\gamma}(r_{2}-l_{2})^{-\gamma}
 |l_{1}-l_{2}|^{2H-2}|r_{1}-r_{2}|^{2H-2}dl_{1}dl_{2}dr_{1}dr_{2}.
\end{multline}
The right hand side of \eqref{Du bound} can be estimated by elementary calculus similarly to \eqref{expression of 2nd z1}-\eqref{desired result}. We let the patient reader check that whenever $\ga<2H-1$ we have
\begin{equation}\label{Du bound 1}
\mathbb{E}\left[\|D^{i}u^{\tau,i}_{ts}\|^{2}_{\mathcal{H}^{\otimes 2}}\right]\lesssim  \left[|\tau-t|^{-2\gamma}|t-s|^{4H-2\gamma}\right]\wedge |t-s|^{4H-4\gamma}.
\end{equation} 
Eventually plugging \eqref{Du bound 1} and \eqref{J1 bound 1} into \eqref{upper bound J1 1}, we end up with
\begin{equation}\label{upper bound J1}
J_{1}\lesssim  \left[|\tau-t|^{-2\gamma}|t-s|^{4H-2\gamma}\right]\wedge |t-s|^{4H-4\gamma}.
\end{equation}
Next we upper bound $J_{2}$ as given in \eqref{def J2}. Recalling that we have computed $D^{i}_{l}(u^{\tau,i}_{ts}(r))$ in \eqref{Du 1} and plugging this identity into \eqref{def J2}, we obtain
\begin{equation}\label{a0}
J_{2}=\left(H(2H-1)\right)^{2}\left(\int_{s}^{t}\int_{s}^{r}(\tau-r)^{-\gamma}(r-l)^{-\gamma}|r-l|^{2H-2}dr\,dl\right)^{2}.
\end{equation}
Along the same lines as for the computations from \eqref{expression of 2nd z1 1} to \eqref{bound 2}, and recalling the fact that $\gamma<2H-1<1$, we get the following upper bound for $J_{2}$,  
\begin{equation}\label{upper bound J2}
J_{2}\lesssim \left[|\tau-t|^{-2\gamma}|t-s|^{4H-2\gamma}\right]\wedge |t-s|^{4H-4\gamma}.
\end{equation}
Eventually plugging \eqref{upper bound J2} and \eqref{upper bound J1} into \eqref{second moment of stratonovich} and recalling again the definition \eqref{psi 1} of $\psi^{1}$, we get
\begin{equation*}
\mathbb{E}\left[\left(\int_{s}^{t} u_{r} d B^{i}_{r}\right)^{2}\right]\lesssim  \left[|\tau-t|^{-2\gamma}|t-s|^{4H-2\gamma}\right]\wedge |t-s|^{4H-4\gamma}=\left|\psi^{1}_{(2H-\gamma,\gamma)}(\tau,t,s)\right|^{2}.
\end{equation*}
This completes the proof of Step 2.

Eventually, combining Step 1 and Step 2, relation \eqref{prob reg} holds for the increment $\bfz^{2,\tau}_{ts}$ as given in~\eqref{z2 ij}-\eqref{z2 ii}. This concludes the proof of \eqref{prob reg}. 
\end{proof}

\begin{remark}
The condition $\gamma<2H-1$, as stated in~\eqref{H, gamma, eta, zeta conditions},  is only invoked in order to properly bound the right hand side of~\eqref{a0}.
\end{remark}

\subsection{Properties of the Volterra rough path family $\{\bfz^{1,\tau},\bfz^{2,\tau}\}$}\label{Sec: prop of volterra RP}
We have constructed a Volterra rough path family $\{\bfz^{1,\tau},\bfz^{2,\tau}\}$ and we have also upper bounded their moment in Section~\ref{first level of the Volterra rough path} and Section \ref{second level of the Volterra rough path}. In this section, we will verify that the family $\{\bfz^{1,\tau},\bfz^{2,\tau}\}$ satisfies Definition \ref{hyp 2}. To this aim, we start by introducing the following notation:

\begin{notation}\label{def of An}
Let  $H\in(1/2, 1)$, and consider two parameters  $\alpha \in(0,H)$, $\gamma\in(0,2H-1)$, and a family $(\eta_{k}, \zeta_{k})$ for $1\le k\le N$ such that \eqref{condition parameters 2} is satisfied. The analytic property \ref{hyp 2}\ref{zj space} of Volterra rough path is also labeled in the following way:
\begin{eqnarray}\label{zj space wrt A}
{\bf z}^{j}\in \bigcap_{(\eta, \zeta)\in \mathcal{A}_N}\mathcal{V}^{(j\rho+\gamma,\gamma,\eta,\zeta)},
\end{eqnarray}
where $\mathcal{A}_N$ is given by 
\begin{eqnarray}\label{def of A}
\mathcal{A}_N=\{(\eta_k,\zeta_k); 1\le k\le N\}. 
\end{eqnarray}
 \end{notation}
Let us first check the analytic part of Definition \ref{hyp 2} for $\bfz^{1,\tau}$.
\begin{prop}\label{z1 space}
Let $H\in (1/2, 1)$, and consider two parameters $\alpha \in(0,H)$, $\gamma\in(0,2H-1)$ and a family $(\eta_k, \zeta_k)$ as in Notation \ref{def of An}.  Then the increment $\bfz^{1,\tau}$ introduced in Definition \ref{def of z1} is almost surely in the Volterra space $\mathcal{V}^{(\alpha,\gamma,\eta,\zeta)}(\Delta_{3};\mathbb{R}^{m})$ for any $(\eta,\zeta)\in \mathcal{A}_N$, where $\mathcal{A}_N$ is given in~\eqref{def of A} and $\mathcal{V}^{(\alpha,\gamma,\eta,\zeta)}(\Delta_{3};\mathbb{R}^{m})$ is introduced in Definition \ref{Volterra space}. In addition,  for any $p\ge 1$ and $(\eta,\zeta)\in \mathcal{A}_N$,  we have that 
\begin{equation}\label{2p moment z1}
\mathbb{E}\left[\|\bfz^{1}\|_{(\alpha,\gamma,\eta,\zeta)}^{2p}\right]<\infty.
\end{equation}
\end{prop}
\begin{proof}
In this proof, we will turn to Proposition \ref{Volterra GRR inequality} in order to prove \eqref{2p moment z1}. According to the definition \eqref{Volterra norm} of Volterra norms, it suffices to show that $\mathbb{E}[\|\bfz^{1}\|^{2p}_{(\alpha,\gamma),1}]$ and $\mathbb{E}[\|\bfz^{1}\|^{2p}_{(\alpha,\gamma,\eta,\zeta),1,2}]$ are finite. We will separate the study of those two moments.

\noindent
\textit{Step 1: Estimate for the $2p$ moment of 1-norm.}
Let us first upper bound $\mathbb{E}[\|\bfz^{1}\|^{2p}_{(\alpha,\gamma),1}]$. Towards this aim, consider a fixed Volterra exponent $\gamma<\alpha<H$ and a parameter $p>1$ to be determined later on. Relation~\eqref{eq:bound 1} is then equivalent to 
\begin{equation}\label{z1 volterra bound}
\|\bfz^{1}\|^{2p}_{(\alpha,\gamma),1}\lesssim \left(U^T_{(\alpha,\gamma),1,p}(\bfz^{1})\right)^{2p}+\|\delta \bfz^{1}\|^{2p}_{(\alpha,\gamma),1}.
\end{equation}
Let us handle the term $\|\delta \bfz^{1}\|^{2p}_{(\alpha,\gamma),1}$ in the right hand side of \eqref{z1 volterra bound}. Gathering the definitions in~\eqref{z1 expression fbm} and~\eqref{delta}, for $(s,m,t)\in \Delta_{3}$ we have 
\begin{equation}\label{delta z1}
\delta_{m}\bfz^{1,\tau,i}_{ts}=B^{i}(h^{\tau}_{ts})-B^{i}(h^{\tau}_{tm})-B^{i}(h^{\tau}_{ms}).
\end{equation}
Moreover recalling that $h^{\tau}_{ts}(r)=(\tau-r)^{-\gamma}\mathbbm{1}_{[s,t]}(r)$, it is readily checked that $h^{\tau}_{ts}-h^{\tau}_{tm}-h^{\tau}_{ms}=0$. We thus get $\delta \bfz^{1,\tau,i}=0$ and \eqref{z1 volterra bound} is reduced to  
\begin{equation}\label{z1 volterra bound 1}
\|\bfz^{1}\|^{2p}_{(\alpha,\gamma),1}\lesssim \left(U^T_{(\alpha,\gamma),1,p}(\bfz^{1})\right)^{2p}.
\end{equation}
Taking expectations on both sides of \eqref{z1 volterra bound 1} and recalling the definition \eqref{U1} of $U^T_{(\alpha,\gamma),1,p}$, we obtain
\begin{equation}\label{z1 GRR 1}
\mathbb{E}\left[\|\bfz^{1}\|^{2p}_{(\alpha,\gamma),1}\right]\lesssim \mathbb{E}\left[ \int_{(v,w)\in \Delta^{\tau}_{2}}\frac{|z^{1,\tau}_{wv}|^{2p}}{|\psi^{1}_{(\alpha,\gamma)}(\tau,w,v)|^{2p}|w-v|^{2}}\,dv\,dw\right].
\end{equation}
Invoking Fubini's theorem and the fact that $\bfz^{1,\tau}_{wv}$ is a Gaussian random variable, we thus get 
\begin{equation}\label{z1 GRR}
\mathbb{E}\left[\|\bfz^{1}\|^{2p}_{(\alpha,\gamma),1}\right]\lesssim \int_{(v,w)\in \Delta^{\tau}_{2}}\frac{\mathbb{E}\left[|z^{1,\tau}_{wv}|^{2}\right]^{p}}{|\psi^{1}_{(\alpha,\gamma)}(\tau,w,v)|^{2p}|w-v|^{2}}dvdw.
\end{equation}
We can now apply \eqref{2nd moment of z1}, and hence relation \eqref{z1 GRR} reads
\begin{equation}\label{z1 GRR 1}
\mathbb{E}\left[\|\bfz^{1}\|^{2p}_{(\alpha,\gamma),1}\right]\lesssim \int_{(v,w)\in \Delta^{\tau}_{2}}\frac{|\psi^{1}_{(H,\gamma)}(\tau,w,v)|^{2p}}{|\psi^{1}_{(\alpha,\gamma)}(\tau,w,v)|^{2p}|w-v|^{2}}dvdw. 
\end{equation}
Recalling the definition \eqref{psi 1} of $\psi^{1}$, we obtain
\begin{equation}\label{z1 2p bound}
\mathbb{E}\left[\|\bfz^{1}\|^{2p}_{(\alpha,\gamma),1}\right]\lesssim \int_{(v,w)\in\Delta^{\tau}_{2}}\frac{\left[|\tau-w|^{-2p\gamma}|w-v|^{2pH}\right]\wedge|w-v|^{2p(H-\gamma)}}{\left(\left[|\tau-w|^{-2p\gamma}|w-v|^{2p\alpha}\right]\wedge|w-v|^{2p(\alpha-\gamma)}\right)|w-v|^{2}}dvdw. 
\end{equation}
In order to upper bound the right hand side of \eqref{z1 2p bound}, we split set $\Delta^{\tau}_{2}$ into two subsets 
\[
E_{1}=\left\{\left(v,w\right)\in\Delta^{\tau}_{2}\big|\, |\tau-w|\le |w-v|\right\},
\quad\text{and}\quad  
E_{2}=\left\{\left(v,w\right)\in\Delta^{\tau}_{2}\big|\, |\tau-w|> |w-v|\right\}.
\]
Then relation \eqref{z1 2p bound} is equivalent to 
\begin{equation}\label{split U z1}
 \mathbb{E}\left[\|\bfz^{1}\|^{2p}_{(\alpha,\gamma),1}\right]\lesssim I_{1}+I_{2},
 \end{equation}
where $I_{1}$ and $I_{2}$ are respectively given by
\begin{align}
&I_{1}=\int_{E_{1}}\frac{\left[|\tau-w|^{-2p\gamma}|w-v|^{2pH}\right]\wedge|w-v|^{2p(H-\gamma)}}{\left[|\tau-w|^{-2p\gamma}|w-v|^{2p\alpha+2}\right]\wedge|w-v|^{2p(\alpha-\gamma)+2}}dvdw,\label{I1 z1}
\\
&I_{2}=\int_{E_{2}}\frac{\left[|\tau-w|^{-2p\gamma}|w-v|^{2pH}\right]\wedge|w-v|^{2p(H-\gamma)}}{\left[|\tau-w|^{-2p\gamma}|w-v|^{2p\alpha+2}\right]\wedge|w-v|^{2p(\alpha-\gamma)+2}}dvdw.\label{I2 z1}
\end{align}
In the following, we will estimate $I_{1}$ and $I_{2}$ separately. 

In order to upper bound $I_{1}$, we first note that for any $(v,w)\in E_{1}$, we have $|\tau-w|\le |w-v|$. Thus 
\begin{equation}\label{inequality 1}
|w-v|^{2p(H-\gamma)}=|w-v|^{2pH}|w-v|^{-2p\gamma}\le |\tau-w|^{-2p\gamma}|w-v|^{2pH},
\end{equation}
and we trivially get 
\begin{equation}\label{integrand n}
\left[|\tau-w|^{-2p\gamma}|w-v|^{2pH}\right]\wedge|w-v|^{2p(H-\gamma)}=|w-v|^{2p(H-\gamma)}.
\end{equation}
In the same way, on $E_{1}$ we can write 
 \begin{equation}\label{integrand d}
\left[|\tau-w|^{-2p\gamma}|w-v|^{2p\alpha+2}\right]\wedge|w-v|^{2p(\alpha-\gamma)+2}=|w-v|^{2p(\alpha-\gamma)+2}.
\end{equation}
Plugging \eqref{integrand n} and \eqref{integrand d} into \eqref{I1 z1}, we get
\begin{equation}\label{E1 z1}
I_{1}=\int_{E_{1}}\frac{|w-v|^{2p(H-\gamma)}}{ |w-v|^{2p(\alpha-\gamma)+2}}= \int_{E_{1}}|w-v|^{2p(H-\alpha)-2}dvdw.
\end{equation}
Similarly, reverting the inequality in \eqref{inequality 1} we get that
\begin{equation}\label{E2 z1}
I_{2}=\int_{E_{2}}\frac{|\tau-w|^{-2p\gamma}|w-v|^{2pH}}{|\tau-w|^{-2p\gamma}|w-v|^{2p\alpha+2}}dvdw=\int_{E_{2}}|w-v|^{2p(H-\alpha)-2}dvdw.
\end{equation}
Now gathering \eqref{E1 z1} and \eqref{E2 z1} into \eqref{split U z1}, we end up with 
\begin{equation}\label{z1 1-norm bound}
 \mathbb{E}\left[\|\bfz^{1}\|^{2p}_{(\alpha,\gamma),1}\right]\lesssim \int_{(v,w)\in\Delta^{\tau}_{2}}|w-v|^{2p(H-\alpha)-2}dvdw. 
\end{equation}
The right hand side above is easily checked to be finite as long as $\alpha<H-\frac{1}{2p}$. 

\noindent
\textit{Step 2: Estimate for the $2p$ moment of $(1,2)$-norm.}
Next, we will show that $\mathbb{E}[\|\bfz^1\|_{(\alpha,\gamma,\eta,\zeta),1,2}]$ is finite. Similarly to the proof for the 1-norm in Step 1, considering again $p\ge 1$. Then resorting to~\eqref{eq:bound 2}, for $(\eta,\zeta)\in \mathcal{A}_N$ we get 
\begin{equation*}
\mathbb{E}\left[\|\bfz^{1}\|^{2p}_{(\alpha,\gamma,\eta,\zeta),1,2}\right]\lesssim \mathbb{E}\left[\left(U^T_{(\alpha,\gamma,\eta,\zeta),1,2,p}(\bfz^{1})\right)^{2p}\right].
\end{equation*}
As in Step 1, recalling the definition \eqref{U12} of $U^T_{(\alpha,\gamma,\eta,\zeta),1,2,p}(\bfz)$, invoking Fubini's theorem and thanks to the fact that $\bfz^{1,\tau \tau^{\prime}}$ is a Gaussian random variable, we obtain
\begin{equation}\label{z1 GRR 12}
\mathbb{E}\left[\|\bfz^{1}\|^{2p}_{(\alpha,\gamma,\eta,\zeta),1,2}\right]\lesssim \int_{(v,w,r^{\prime},r)\in\Delta^{\tau}_{4}}\frac{\mathbb{E}^{p}\left[|z^{1,rr^{\prime}}_{wv}|^{2}\right]}{|\psi^{1,2}_{(\alpha,\gamma,\eta,\zeta)}(r,r^{\prime},w,v)|^{2p}|w-v|^{2}|r-r^{\prime}|^{2}}\,dvdwdr^{\prime}dr.
\end{equation}
In addition, owing to \eqref{2nd moment of z1 2}, relation \eqref{z1 GRR 12} yields
\begin{equation}\label{z1 GRR 12 1}
\mathbb{E}\left[\|\bfz^{1}\|^{2p}_{(\alpha,\gamma,\eta,\zeta),1,2}\right]\lesssim \int_{(v,w,r^{\prime},r)\in\Delta^{\tau}_{4}}\frac{|\psi^{1,2}_{(H,\gamma,\eta+\frac{1}{p},\zeta+\frac{1}{p})}(r,r^{\prime},w,v)|^{2p}}{|\psi^{1,2}_{(\alpha,\gamma,\eta,\zeta)}(r,r^{\prime},w,v)|^{2p}|w-v|^{2}|r-r^{\prime}|^{2}} \, dvdwdr^{\prime}dr.
\end{equation}
We now recall the definition \eqref{psi 12} of $\psi^{1,2}$ and plug this identity into \eqref{z1 GRR 12 1}. We get 
 \begin{equation}\label{z1 GRR 12 2}
\mathbb{E}\left[\|\bfz^{1}\|^{2p}_{(\alpha,\gamma,\eta,\zeta),1,2}\right]\lesssim \int_{(v,w,r^{\prime},r)\in\Delta^{\tau}_{4}}
g_{(H,\alpha,\gamma,\eta,\zeta,p)}(r,r^{\prime},w,v) \, dvdwdr^{\prime}dr,
\end{equation}
where $g_{(H,\alpha,\gamma,\eta,\zeta,p)}(r,r^{\prime},w,v)$ is given by
\begin{align}\label{expression g}
&g_{(H,\alpha,\gamma,\eta,\zeta,p)}(r,r^{\prime},w,v)=
\\ 
&\frac{|r-r^{\prime}|^{2p(\eta+\frac{1}{p})}|r^{\prime}-w|^{-2p(\eta-\zeta)}\left(\left[|r^{\prime}-w|^{-2p(\gamma+\zeta+\frac{1}{p})}|w-v|^{2pH}\right]\wedge|w-v|^{2p(H-\gamma-\zeta-\frac{1}{p})}\right)}{|r-r^{\prime}|^{2p\eta}|r^{\prime}-w|^{-2p(\eta-\zeta)}\left(\left[|r^{\prime}-w|^{-2p(\gamma+\zeta)}|w-v|^{2p\alpha}\right]\wedge|w-v|^{2p(\alpha-\gamma-\zeta)}\right)|w-v|^{2}|r-r^{\prime}|^{2}}. \notag
\end{align}
Thanks to cancellations, we can simplify the right hand side of \eqref{expression g} as 
\begin{equation}\label{expression g 1}
g_{(H,\alpha,\gamma,\eta,\zeta,p)}(r,r^{\prime},w,v)=\frac{\left[|r^{\prime}-w|^{-2p(\gamma+\zeta)-2}|w-v|^{2pH}\right]
\wedge
|w-v|^{2p(H-\gamma-\zeta)-2}}{\left(\left[|r^{\prime}-w|^{-2p(\gamma+\zeta)}|w-v|^{2p\alpha}\right]
\wedge|w-v|^{2p(\alpha-\gamma-\zeta)}\right)|w-v|^{2}}.
\end{equation}
Plugging \eqref{expression g 1} into \eqref{z1 GRR 12 2}, we thus get 
\begin{multline}\label{z1 GRR 12 3}
\mathbb{E}\left[\|\bfz^{1}\|^{2p}_{(\alpha,\gamma,\eta,\zeta),1,2}\right]\\
\lesssim \int_{(v,w,r^{\prime},r)\in\Delta^{\tau}_{4}}
\frac{\left[|r^{\prime}-w|^{-2p(\gamma+\zeta)-2}|w-v|^{2pH}\right]\wedge|w-v|^{2p(H-\gamma-\zeta)-2}}{\left(\left[|r^{\prime}-w|^{-2p(\gamma+\zeta)}|w-v|^{2p\alpha}\right]\wedge|w-v|^{2p(\alpha-\gamma-\zeta)}\right)|w-v|^{2}} \,dvdwdr^{\prime}dr .
\end{multline}
Notice that the right hand side of \eqref{z1 GRR 12 3} is now very similar to the right hand side of \eqref{z1 2p bound}. Therefore with the same steps as for \eqref{z1 2p bound}-\eqref{z1 1-norm bound}, we obtain that
\begin{equation}\label{z1 12-norm bound}
 \mathbb{E}\left[\|\bfz^{1}\|^{2p}_{(\alpha,\gamma,\eta,\zeta),1,2}\right]\lesssim \int_{(v,w,r^{\prime},r)\in\Delta^{\tau}_{4}}|w-v|^{2p(H-\alpha)-4}dvdwdr^{\prime}dr<\infty.
\end{equation}
The right hand side of \eqref{z1 12-norm bound} is finite as long as $p>\frac{3}{2}(H-\alpha)^{-1}$, or equivalently $\alpha<H-\frac{3}{2p}$. 

\noindent
\textit{Step 3: \eqref{2p moment z1} holds for any $p\ge 1$.} Invoking \eqref{z1 volterra bound} we immediately have 
\begin{equation}\label{some 2p moment z1}
\mathbb{E}\left[\|\bfz^{1}\|^{2p}_{(\alpha,\gamma,\eta,\zeta)}\right]\lesssim\mathbb{E}\left[\|\bfz^{1}\|^{2p}_{(\alpha,\gamma),1}\right]+\mathbb{E}\left[\|\bfz^{1}\|^{2p}_{(\alpha,\gamma,\eta,\zeta),1,2}\right].
\end{equation}
Furthermore, combining \eqref{z1 1-norm bound} and \eqref{z1 12-norm bound} in the right hand side of \eqref{some 2p moment z1}, we end up with 
\begin{equation}\label{2p moment z1 copy}
\mathbb{E}\left[\|\bfz^{1}\|^{2p}_{(\alpha,\gamma,\eta,\zeta)}\right]<\infty, \quad \text{for any}\quad p > \frac{3}{2(H-\al)}. 
\end{equation}
Next we observe that if \eqref{2p moment z1 copy} is satisfied under the constraint $p > \frac{3}{2(H-\al)}$, it is also verified for all $p\ge 1$.
This yields the desired result \eqref{2p moment z1}. Moreover, it is easy to check that \eqref{2p moment z1 copy} implies 
\begin{equation*}
 \|\bfz^{1}\|_{(\alpha,\gamma,\eta,\zeta)}<\infty\quad \text{a.s.}
 \end{equation*}
This means that $\bfz^{1}$ is almost surely in the Volterra space $\mathcal{V}^{(\alpha,\gamma,\eta,\zeta)}(\Delta_{3};\mathbb{R}^{m})$.
\end{proof}

\begin{remark}
Note that the Volterra sewing lemma in \cite{HarangTindel} could indeed be used to construct $\bfz^1$ in a purely pathwise manner due to the H\"older regularity of the fBm combined with the assumption $H-\gamma>0$. This would then be constructed as the  pathwise integral  given by 
\begin{equation*}
\bfz^{1,\tau}_{ts}=\lim_{|\cp|\rightarrow 0} \sum_{[u,v]\in \cp} k(\tau,u)B_{vu}(\omega), 
\end{equation*}
where the sum converges by deterministic arguments and $\cp$ is a partition of $[s,t]$. In fact, it follows then directly that $\bfz^1\in \cv^{(\alpha,\gamma,\eta,\zeta)}$ for any  $ \alpha-\gamma>0$ and  $(\eta,\zeta)\in \ca_N$. 
However we have chosen to construct $\bfz^1$ by probabilistic means, our motivation being twofold: 
\begin{enumerate}[wide, labelwidth=!, labelindent=0pt, label=(\roman*)]
\setlength\itemsep{.03in}

\item
As the reader will see, the construction of the second order integral $\bfz^2$ is probabilistic by nature. Therefore it is natural and more consistent to construct $\bfz^1$ with the same kind of method.

\item
The probabilistic construction of $\bfz^1$ is not only instructive, but also provides useful probabilistic bounds for the moments of $\bfz^1$. Those estimates are of interest on their own.

\end{enumerate}
\end{remark}

With Proposition \ref{z1 space} in hand, we finish the study of $\bfz^{1}$ by proving the algebraic relation \eqref{hyp b} for $\bfz^{1,\tau}$ in more detail.
\begin{prop}\label{almost surely delta z1}
The increment $\bfz^{1,\tau,i}_{ts}$ as given in \eqref{z1 expression fbm} satisfies relation \eqref{hyp b}, that is almost surely we have
\begin{equation}\label{prob chen z1}
\delta_{m}\bfz^{1,\tau,i}_{ts}=0,\quad \text{for all}\quad (s,m,t,\tau)\in \Delta_{4} .
\end{equation}
\end{prop}
\begin{proof}
For fixed $(s,m,t,\tau)\in \Delta_{4}$, we have obtained in \eqref{delta z1} that $\delta_{m}\bfz^{1,\tau,i}_{ts}=0$ almost surely. We will now prove that 
\begin{equation}\label{mapping z1}
(t,\tau)\in\Delta_{2}\mapsto \bfz^{1,\tau}_{t}\in \mathbb{R}^{m} \,\,\text{is a continuous function}.
\end{equation}
By a standard argument, which consists in taking limits on rational points, this will achieve our claim \eqref{prob chen z1}. 
The proof of \eqref{mapping z1} relies on Lemma \ref{z1 second moment}. Indeed, according to \eqref{2nd moment of z1} for $(s,t,\tau)$ in $\Delta_{3}$, we have 
\begin{equation}\label{E z1 1}
\mathbb{E}\left[\left(\bfz^{1,\tau}_{ts}\right)^{2}\right]\lesssim \left|t-s\right|^{2(H-\gamma)}.
\end{equation}
In the same way thanks to \eqref{2nd moment of z1 2} applied with $\zeta=\eta=H-\gamma-\epsilon$ with a small $\epsilon>0$, we get
\begin{equation}\label{E z1 12}
\mathbb{E}\left[\left(\bfz^{1,\tau\tau^{\prime}}_{ts}\right)^{2}\right]\lesssim \left|\tau-\tau^{\prime}\right|^{H-\gamma-\epsilon}\left|t-s\right|^{\epsilon}.
\end{equation}
Gathering \eqref{E z1 1} and \eqref{E z1 12}, we end up with the following inequality, valid for $(s,t,\tau^{\prime},\tau)\in \Delta_{4}$:
\begin{equation}\label{E z1 122}
\|  \bfz^{1,\tau\tau^{\prime}}_{ts} \|_{L^{2}(\Omega)}
\lesssim \left|\tau-\tau^{\prime}\right|^{H-\gamma-\epsilon}+\left|t-s\right|^{H-\gamma}.
\end{equation}
Moreover $\bfz^{1,\tau\tau^{\prime}}_{ts}$ is a Gaussian random variable. Hence the upper bound \eqref{E z1 122} can be extended to arbitrary norms in $L^{p}(\Omega)$. Therefore a standard application of Kolmogorov's criterion yields the continuity property \eqref{E z1 1} for $\bfz^{1,\tau}$. This finishes our proof.  
\end{proof}

We now turn to the analysis $\bfz^{2,\tau}$. We start this study by verifying the algebraic relation \eqref{hyp b} for $\bfz^{2,\tau}$. 
\begin{prop}\label{delta z2 relation}
The increment $\bfz^{2,\tau}_{ts}$ as given in \eqref{z2 ij}-\eqref{z2 ii} satisfies relation \eqref{hyp b}, that is 
\begin{equation}\label{prob chen}
\delta_{m}\bfz^{2,\tau,i,j}_{ts}=\bfz^{1,\tau,j}_{tm}\ast \bfz^{1,\cdot,i}_{ms}, \quad \text{for all}\,\,(s,m,t,\tau)\in \Delta_{4} \quad \text{a.s.}
\end{equation}
\end{prop}
\begin{proof}
In order to show \eqref{prob chen}, we first prove that \eqref{prob chen} holds for fixed $(s,m,t)\in\Delta^{\tau}_{3}$. According to Definition \ref{the second level fBm Volterra rough path}, we will separate the proof into two cases $i\neq j$ and $i=j$. 

\noindent
\textit{Step 1: \eqref{prob chen} holds for fixed $(s,m,t)\in\Delta^{\tau}_{3}$ when $i\neq j$.} In this step, let us handle the case $i\neq j$. For any $(s,m,t)\in \Delta^{\tau}_{3}$, gathering \eqref{z2 ij} and \eqref{delta}, we have 
\begin{equation}\label{delta z2 ij expression}
\delta_{m} \bfz^{2,\tau,i,j}_{ts}= B^{j}\left(u^{\tau,i}_{ts}\right)-B^{j}\left(u^{\tau,i}_{tm}\right)-B^{j}\left(u^{\tau,i}_{ms}\right),
\end{equation}
where we recall that the process $u$ is defined by \eqref{def of u}. In order to calculate the right hand side of \eqref{delta z2 ij expression}, it is thus sufficient to compute $\delta_{m}u^{\tau,i}_{ts}=u^{\tau,i}_{ts}-u^{\tau,i}_{tm}-u^{\tau,i}_{ms}$. To this aim, according to the definition \eqref{def of u} of $u^{\tau,i}$ , we obtain 
\begin{eqnarray*}
\delta_{m}u^{\tau,i}_{ts}(r)&=&u^{\tau,i}_{ts}(r)-u^{\tau,i}_{tm}(r)-u^{\tau,i}_{ms}(r)
\\
&=&
\left(\tau-r\right)^{-\gamma}\bfz^{1,r,i}_{rs}\mathbbm{1}_{[s,t]}(r)-\left(\tau-r\right)^{-\gamma}\bfz^{1,r,i}_{rm}\mathbbm{1}_{[m,t]}(r)-\left(\tau-r\right)^{-\gamma}\bfz^{1,r,i}_{rs}\mathbbm{1}_{[s,m]}(r). 
\end{eqnarray*}
Resorting to the definition \eqref{z1 expression fbm} of $\bfz^{1,r,i}_{rs}$, we thus get
\begin{equation}\label{delta u expression}
\delta_{m}u^{\tau,i}_{ts}(r)=\left(\tau-r\right)^{-\gamma}\left(B^{i}(h^{r}_{rs})\mathbbm{1}_{[s,t]}(r)-B^{i}(h^{r}_{rm})\mathbbm{1}_{[m,t]}(r)-B^{i}(h^{r}_{rs})\mathbbm{1}_{[s,m]}(r)\right),
\end{equation} 
where the expression for $h$ is given in Definition \ref{def of z1}.
The right hand side of \eqref{delta u expression} can be simplified by elementary calculus. We thus let the patient reader check that we have
\begin{equation}\label{u expression}
\delta_{m}u^{\tau,i}_{ts}(r)=\left(\tau-r\right)^{-\gamma}B^{i}\left(h^{r}_{ms}\right)\mathbbm{1}_{[m,t]}(r).
\end{equation}
Furthermore, according to the definition of $h$ in Definition \ref{def of z1}, we have that $\left(\tau-r\right)^{-\gamma}\mathbbm{1}_{[m,t]}(r)=h^{\tau}_{tm}(r)$. Hence \eqref{u expression} can be recast as  
\begin{equation}\label{u expression 1}
\delta_{m}u^{\tau,i}_{ts}(r)=h^{\tau}_{tm}(r)B^{i}\left(h^{r}_{ms}\right).
\end{equation}
Plugging \eqref{u expression 1} into \eqref{delta z2 ij expression}, we thus have
\begin{equation}\label{delta z2 BjBi}
\delta_{m} \bfz^{2,\tau,i,j}_{ts}=B^{j}\left(h^{\tau}_{tm}B^{i}\left(h^{r}_{ms}\right)\right).
\end{equation}
Resorting to the property \eqref{def of wiener} of $B^{j}(h)$, the right hand side of \eqref{delta z2 BjBi} can be written as 
\begin{equation}\label{delta z2 BjBi Rie sum}
\delta_{m} \bfz^{2,\tau,i,j}_{ts}=\lim_{|\mathcal{P}|\to 0}\sum_{[r,v]}
B^{j}_{vr}\, h^{\tau}_{tm}(r)\, B^{i}(h^{r}_{ms}),
\end{equation}
where we recall that $\mathcal{P}$ is a generic partition of $[m,t]$ whose mesh $|\mathcal{P}|$ is converging to $0$, and where the limit holds in $L^{2}(\Omega)$. We now consider a subsequence of partitions in order to get an almost sure convergence in \eqref{delta z2 BjBi Rie sum}. According to the definition \eqref{convolution in 1d} of convolution product we end up with 
\begin{equation*}
\delta_{m} \bfz^{2,\tau,i,j}_{ts}
=\bfz^{1,\tau,j}_{tm}\ast\bfz^{1,\cdot,i}_{ms},
\end{equation*}
where we have used the definition \eqref{z1 expression fbm} of $\bfz^{1,\tau}$.

\noindent
\textit{Step 2: \eqref{prob chen} holds for fixed $(s,m,t)\in\Delta^{\tau}_{3}$ when $i=j$.} In this step, we will deal with the case $i=j$ for the the second level of the Volterra rough path. For any $(s,m,t)\in \Delta^{\tau}_{3}$, according to the definition \eqref{z2 ii} of $\bfz^{2,\tau,i,i}_{ts}$, we obtain
\begin{equation}\label{delta z2 ii}
\delta_{m}\bfz^{2,\tau,i,i}_{ts}=\delta_{m}\left(\int^{t}_{s}u^{\tau,i}_{ts}(r) dB^{i}_{r}\right),
\end{equation}
where the integral above is understood in the Stratonovich sense. According to \eqref{def of stra}, we have 
\begin{equation*}
\int_{s}^{t}u^{\tau,i}_{ts}(r)dB^{i}_{r}=\lim_{|\mathcal{P}|\to 0}S^{i,\mathcal{P}}_{ts},
\quad\text{where}\quad
S^{i,\mathcal{P}}_{ts}=\int^{t}_{s}u^{\tau,i}_{ts}(r) B^{i,\mathcal{P}}_{r}dr.
\end{equation*}
Now for a fixed $\mathcal{P}$, elementary algebraic manipulations show that 
\begin{equation}\label{delta S}
\delta_{m}S^{i,\mathcal{P}}_{ts}=\bfz^{1,\tau,i,\mathcal{P}}_{tm}\ast\bfz^{1,\cdot,i}_{ms},
\end{equation}
where $\bfz^{1,\tau,j,\mathcal{P}}_{tm}$ is defined by \eqref{def of stra}. Taking limits on both sides of \eqref{delta S} as $\mathcal{P}\to 0$, we get 
\begin{equation*}
\delta_{m}\bfz^{2,\tau,i,i}_{ts}=B^{i}\left(h^{\tau}_{tm}B^{i}(h^{r}_{ms})\right)=\bfz^{1,\tau,i}_{tm}\ast \bfz^{1,\cdot,i}_{ms},
\end{equation*}
which proves \eqref{prob chen} for $i=j$.

\noindent
\textit{Step 3: \eqref{prob chen} holds for all $(s,m,t)\in \Delta^{\tau}_{3}$.}
The proof of this fact, based on Kolmogorov's criterion for continuity of stochastic processes, is very similar to  the considerations in Proposition \ref{almost surely delta z1}. For sake of conciseness, it is omitted here. The proof of \eqref{prob chen} is now complete. 
\end{proof}
Using the knowledge gained from Proposition \ref{delta z2 relation}, we are now ready to check the regularity of the object $\delta \bfz^{2,\tau}$. 
\begin{prop}\label{delta z2 space}
Let $H\in (\frac{1}{2}, 1)$, and  consider the second level $\bfz^{2,\tau}$ of the Volterra rough path, as defined in 
\eqref{z2 ij}-~\eqref{z2 ii}. Recall that $\delta \bfz^{2,\tau}$ is  defined on $\Delta_{4}$, and we refer to Definition \ref{def of delta norm} for the definition of $\mathcal{V}^{(\alpha,\gamma,\eta,\zeta)}(\Delta_{4};\mathbb{R}^{m})$. Consider four parameters $\alpha \in(0,H)$, $\gamma\in(0,2H-1)$ and $\eta,\, \zeta\in [0,1]$,  satisfying relation \eqref{condition parameters 2}. Let also $\mathcal{A}_N$ be the set defined by \eqref{def of A}. Then almost surely, for all $(\eta, \zeta)\in \mathcal{A}_N$ we have
\begin{equation} \label{delta z2 volterra space}
\delta \bfz^{2,\tau}\in \mathcal{V}^{(2\rho+\gamma,\gamma,\eta,\zeta)}(\Delta_{4};\mathbb{R}^{m}),
\end{equation}
where we recall that $\rho=\alpha-\gamma$. 
Moreover, for all $p\ge1$ we have 
\begin{equation}\label{2p moment delta z}
\mathbb{E}\left[\|\delta \bfz^{2,\tau}\|_{(2\rho+\gamma,\gamma,\eta,\zeta)}^{2p}\right]<\infty,
\end{equation}
where the norm above is understood as in \eqref{Volterra norm delta}.
\end{prop}
\begin{proof}
In this proof, we will show that \eqref{2p moment delta z} holds for any $p\ge 1$, and it is easy to check that \eqref{delta z2 volterra space} is a direct consequence of \eqref{2p moment delta z}. According to the definition \eqref{Volterra norm delta}, it is necessary to prove that $\mathbb{E}[\|\delta \bfz^{2,\tau}\|^{2p}_{(2\rho+\gamma,\gamma),1}]$ and $\mathbb{E}[\|\delta \bfz^{2,\tau}\|^{2p}_{(2\rho+\gamma,\gamma,\eta,\zeta),1,2}]$ are finite. Thanks to \eqref{prob chen}, for any $(s,u,t,\tau)\in \Delta_{4}$ we have 
\begin{equation}\label{delta z2 1}
\delta_u \bfz^{2,\tau}_{ts}=\bfz^{1,\tau}_{tu}\ast \bfz^{1,\cdot}_{us}.
\end{equation}
Hence resorting to \eqref{eq:z conv y bound}, we get
\begin{equation}\label{delta z2}
\left|\delta_u \bfz^{2,\tau}_{ts}\right|=\left|\bfz^{1,\tau}_{tu}\ast \bfz^{1,\cdot}_{us}\right|\lesssim \|\bfz^{1}\|_{(\alpha,\gamma),1}\|\bfz^{1}\|_{(\alpha,\gamma,\eta,\zeta),1,2} \,
\psi^{1}_{(2\rho+\gamma,\gamma)}(\tau,t,s).
\end{equation}
Dividing by $\psi^{1}_{2\rho+\gamma,\gamma}(\tau,t,s)$ on both sides of \eqref{delta z2}, and then taking supremum over $(s,u,t,\tau)\in~\Delta_{4}$, we obtain
\begin{equation}\label{delta z2 bound 1}
\|\delta \bfz^{2}\|_{(2\rho+\gamma,\gamma),1}\le  \|\bfz^{1}\|_{(\alpha,\gamma),1}\|\bfz^{1}\|_{(\alpha,\gamma,\eta,\zeta),1,2}\,\,,
\end{equation}
where we have used the definition \eqref{delta zj norm 1} of $1-$norm for the Volterra space $\mathcal{V}^{(2\rho+\gamma,\gamma,\eta,\zeta)}(\Delta_{4};\mathbb{R}^{m})$.\\
Similarly, resorting to \eqref{eq:z conv y bound 2} and \eqref{delta zj norm 12}, for any $(s,u,t,\tau^{\prime},\tau)\in \Delta_{5}$ and $(\eta,\zeta)\in \mathcal{A}_N$. We let the patient reader check that we have
\begin{equation}\label{delta z2 bound 12}
\|\delta \bfz^{2}\|_{(2\rho+\gamma,\gamma,\eta,\zeta),1,2}\le  \|\bfz^{1}\|_{(\alpha,\gamma,\eta,\zeta),1,2}\|\bfz^{1}\|_{(\alpha,\gamma,\eta,\zeta),1,2}.
\end{equation}
Combining \eqref{delta z2 bound 1} and \eqref{delta z2 bound 12}, and recalling the definition \eqref{Volterra norm delta} again, we thus obtain
\begin{equation}\label{delta z2 bound}
\|\delta \bfz^{2}\|_{(2\rho+\gamma,\gamma,\eta,\zeta)}=\|\delta \bfz^{2}\|_{(2\rho+\gamma,\gamma),1}+\|\delta \bfz^{2}\|_{(2\rho+\gamma,\gamma,\eta,\zeta),1,2}\lesssim  \|\bfz^{1}\|^{2}_{(\alpha,\gamma,\eta,\zeta)}.
\end{equation}
Taking $2p$ moments on both sides of \eqref{delta z2 bound}, we thus get
\begin{equation}\label{delta z2 moment bound}
\mathbb{E}\left[\|\delta \bfz^{2}\|^{2p}_{(2\rho+\gamma,\gamma,\eta,\zeta)}\right]\lesssim \mathbb{E}\left[  \|\bfz^{1}\|^{4p}_{(\alpha,\gamma,\eta,\zeta)}\right].
\end{equation}
According to \eqref{2p moment z1}, the right hand side of \eqref{delta z2 moment bound} is finite. This means that we have
\begin{equation}
\mathbb{E}\left[\|\delta \bfz^{2}\|^{2p}_{(2\rho+\gamma,\gamma,\eta,\zeta)}\right]<\infty, \qquad \text{for any }\qquad p\ge1.
\end{equation}
This is the desired result. 
\end{proof}
Finally, let us close this section by giving the proof of the regularity result for $\bfz^{2,\tau}$ .
\begin{prop}\label{z2 space}
Under the same assumption as for Proposition \ref{delta z2 space}, the second level of the Volterra rough path $\bfz^{2,\tau}$ introduced in \eqref{z2 ij}-\eqref{z2 ii} is almost surely an element of the space $\mathcal{V}^{(2\rho+\gamma,\gamma,\eta,\zeta)}(\Delta_{3};\mathbb{R}^{m})$ for any $\alpha,\gamma \in(0,1)$ and $\eta, \zeta\in [0,1]$  satisfying  relation \eqref{condition parameters 2}. Furthermore, for  $H-~\alpha>~\frac{1}{4p}$ and any $(\eta, \zeta)\in \mathcal{A}_N$ (where $\mathcal{A}_N$ is given in \eqref{def of A}), we have that 
\begin{equation}\label{2p moment z2}
\mathbb{E}\left[\|\bfz^{2}\|_{(2\rho+\gamma,\gamma,\eta,\zeta)}^{2p}\right]<\infty.
\end{equation}

\end{prop}
\begin{proof}
Our strategy to prove this Proposition is the same as for the proof of Proposition \ref{z1 space}, that is we will appeal to the Volterra GRR Lemma \ref{Volterra GRR inequality} to show that $ \mathbb{E}[\|\bfz^{2}\|^{2p}_{(2\rho+\gamma,\gamma),1}]$ and $\mathbb{E}[\|\bfz^{2}\|^{2p}_{(2\rho+\gamma,\gamma,\eta,\zeta),1,2}]$ are both finite. Let us first show that $ \mathbb{E}[\|\bfz^{2}\|^{2p}_{(2\rho+\gamma,\gamma),1}]$ is finite. To this aim, consider a fixed Volterra exponent $\alpha\in (\gamma,H)$ and a parameter $p\ge1$ to be determined later. Then relation \eqref{eq:bound 1} reads
\begin{equation}\label{z2 ij bound 1}
 \|\bfz^{2}\|^{2p}_{(2\rho+\gamma,\gamma),1}\lesssim \left(U^T_{(2\rho+\gamma,\gamma),1,p}(\bfz^{2})\right)^{2p} +\|\delta \bfz^{2}\|^{2p}_{(2\rho+\gamma,\gamma),1}. 
\end{equation}
Taking expectations on both sides of \eqref{z2 ij bound 1}, we obtain
\begin{equation}\label{z2 ij bound expectation 1}
\mathbb{E}\left[ \|\bfz^{2}\|^{2p}_{(2\rho+\gamma,\gamma),1}\right]\lesssim \mathbb{E}\left[\left(U^T_{(2\rho+\gamma,\gamma),1,p}(\bfz^{2})\right)^{2p}\right] +\mathbb{E}\left[\|\delta \bfz^{2}\|^{2p}_{(2\rho+\gamma,\gamma),1}\right]\,.
\end{equation}
Recalling \eqref{2p moment delta z}, the second term of the right hand side of \eqref{z2 ij bound expectation 1} is finite. In order to upper bound the left hand side of \eqref{z2 ij bound expectation 1}, it is thus sufficient to estimate the first term $\mathbb{E}[(U^T_{(2\rho+\gamma,\gamma),1,p}(\bfz^{2}))^{2p}]$. Toward this aim, we set 
\begin{equation*}
A_{\ga,\rho,p}=\mathbb{E}\left[\left(U^T_{(2\rho+\gamma,\gamma),1,p}(\bfz^{2})\right)^{2p}\right].
\end{equation*}
Recalling the definition \eqref{U12} of $U^T_{(2\rho+\gamma,\gamma),1,p}$, we have
\begin{equation}\label{def A}
A_{\ga,\rho,p}=\mathbb{E}\left[\int_{(v,w)\in \Delta^{\tau}_{2}}\frac{|z^{2,\tau}_{wv}|^{2p}}{|\psi^{1}_{(2\rho+\gamma,\gamma)}(\tau,w,v)|^{2p}|w-v|^{2}}\,dv\,dw\right].
\end{equation}
Observe that $\bfz^{2,\tau}_{wv}$ is an element of the second chaos of the fBm $B$, on which all $L^{p}$ norms are equivalent. Hence invoking Fubini's theorem, we get 
\begin{equation}\label{A bound 1}
A_{\ga,\rho,p} \lesssim \int_{(v,w)\in \Delta^{\tau}_{2}}\frac{\mathbb{E}^{p}\left[|z^{2,\tau}_{wv}|^{2}\right]}{|\psi^{1}_{(2\alpha-\gamma,\gamma)}(\tau,w,v)|^{2p}|w-v|^{2}}\,dv\,dw.
\end{equation}
We now apply \eqref{prob reg} to the right hand side of \eqref{A bound 1}, we obtain
\begin{equation}\label{z2 GRR}
A_{\ga,\rho,p}\lesssim \int_{(v,w)\in \Delta^{\tau}_{2}}\frac{|\psi^{1}_{(2H-\gamma,\gamma)}(\tau,w,v)|^{2p}}{|\psi^{1}_{(2\rho+\gamma,\gamma)}(\tau,w,v)|^{2p}|w-v|^{2}}\,dvdw. 
\end{equation}
Notice that relation \eqref{z2 GRR} is very similar to \eqref{z1 GRR 1}. Hence we can carry out the same procedure going from \eqref{z1 GRR 1} to \eqref{z1 1-norm bound} in the proof of Proposition \ref{z1 space}. We end up with 
\begin{equation}\label{A bound}
A_{\ga,\rho,p}\lesssim \int_{(v,w)\in \Delta^{\tau}_{2}}|w-v|^{4p(H-\alpha)-2}\,dvdw.
\end{equation} 
Eventually plugging \eqref{A bound} into \eqref{z2 ij bound expectation 1}, we get
\begin{equation}\label{z2 2pmoment bound}
\mathbb{E}\left[ \|\bfz^{2}\|^{2p}_{(2\rho+\gamma,\gamma),1}\right]\lesssim  \int_{(v,w)\in \Delta^{\tau}_{2}}|w-v|^{4p(H-\alpha)-2}\,dvdw+\mathbb{E}\left[\|\delta \bfz^{2}\|^{2p}_{(2\rho+\gamma,\gamma),1}\right].
\end{equation}
Recalling \eqref{2p moment delta z} again, the right hand side above is easily checked to be finite as long as $p>~\frac{1}{4}(H-~\alpha)^{-1}$. Considering such a $p$ (which is allowed since $z^{2,\tau}_{wv}$ admits moments of all orders), we thus obtain  
\begin{equation}\label{z2 2p moment bound 1}
\mathbb{E}\left[ \|\bfz^{2}\|^{2p}_{(2\rho+\gamma,\gamma),1}\right]<\infty.
\end{equation}
Next we will show that $\mathbb{E}[\|\bfz^{2}\|^{2p}_{(2\rho+\gamma,\gamma,\eta,\zeta),1,2}]$ is finite for any $(\eta,\zeta)\in \mathcal{A}_N$. Similarly to the steps going from \eqref{z2 ij bound 1} to \eqref{z2 2pmoment bound}, we resort to \eqref{eq:bound 2} in order to get 
\begin{multline}\label{z2 2p moment bound 12}
\mathbb{E}\left[ \|\bfz^{2}\|^{2p}_{(2\rho+\gamma,\gamma,\eta,\zeta),1,2}\right]\lesssim  \int_{(v,w,r^{\prime},r)\in \Delta^{\tau}_{4}}|w-v|^{4p(H-\alpha)-4}\,dvdwdr^{\prime}dr
\\
+\mathbb{E}\left[\|\delta \bfz^{2,\tau}\|^{2p}_{(2\rho+\gamma,\gamma,\eta+\frac{1}{p},\zeta+\frac{1}{p}),1,2}\right].
\end{multline}
Owing to \eqref{2p moment delta z}, the right hand side of \eqref{z2 2p moment bound 12} is finite as long as $p>\frac{3}{4(H-\alpha)}$. Eventually combining~\eqref{z2 2p moment bound 1} and~\eqref{z2 2p moment bound 12}, and recalling our definition \eqref{Volterra norm} of $(\alpha,\gamma,\eta,\zeta)$-norm, we trivially get that  
\begin{equation}\label{z2 moment bound}
\mathbb{E}\left[\|\bfz^{2}\|^{2p}_{(2\rho+\gamma,\gamma,\eta,\zeta)}\right]<\infty.
\end{equation}
This completes the proof.
\end{proof}
\section{Volterra rough path driven by Brownian motion}\label{sec:VRP for BM}
In this section we will construct Volterra type iterated integrals with respect to a Brownian motion. In this case our stochastic integrals will  be interpreted in the  It\^o sense. The reason for this is that when the singularity exponent $\gamma$ is strictly positive, the  standard It\^o-Stratonovich correction diverge (as will be illustrated in more detail later in this section).  However, in order to take advantage of the computations performed in Section \ref{Volterra rough path driven by fractional Brownian motion}, we will stick to a Malliavin calculus setting. We start by highlighting in Section \ref{Analysis on wiener space} the differences between basic stochastic analysis notions in the fBm context with $H>1/2$ and $H=1/2$ (representing the Brownian motion).

\subsection{Analysis on the Wiener space}\label{Analysis on wiener space}
The Malliavin calculus preliminaries for a Brownian motion are similar to what we wrote in Section \ref{Malliavin calculus preliminaries} for a fBm. Keeping most of our previous notation, let us just highlight the main differences between the two situations. 

\begin{enumerate}[wide, labelwidth=!, labelindent=0pt, label=(\roman*)]

\item
Our notation for the Brownian driving process is $W=(W^{1},\ldots,W^{m})$. The covariance function for each independent component is  $R(s,t)=s\wedge t$.

\smallskip

\item
The space $\mathcal{H}$ is $L^{2}([0,T])$, with inner product
\begin{equation}\label{isometry for W}
\langle f,g \rangle_{\mathcal{H}}=\int^{T}_{0}f_{u}g_{u}du.
\end{equation}

\item
Let $u$ be an adapted process in $L^2([0,T]\times \Omega)$. Then the It\^o integral $\int^{T}_{0}u_{t} \, \delta^{\diamond}W^{j}_{t}$ is well defined for all $j=1, \ldots, m$. It enjoys the It\^o isometry property 
\begin{eqnarray}\label{ito isometry}
\mathbb{E}\left[\left(\int^{T}_{0}u_{t} \, \delta^{\diamond}W^{j}_{t}\right)^2\right]=\int^{T}_{0} \mathbb{E}\left[u^2_{t}\right] dt.
\end{eqnarray}
Observe that for $L^2$-adapted processes, It\^o and Skorokhod's integrals coincide. This explains why we still use the symbol $\delta^{\diamond}$ in the left hand side of \eqref{ito isometry}. 
\end{enumerate}
\subsection{Definition of the Volterra rough path}
In this section we will construct and estimate iterated integrals in case of a driving noise given by a $m$-dimensional Brownian motion $W$. This case is rougher than in Section~\ref{Volterra rough path driven by fractional Brownian motion}, although Volterra stochastics differential equations are arguably already addressed in the classical reference \cite{Nualart}. Nevertheless, it should be noticed that a rough path point of view on equation \eqref{eq:intro Volterra} driven by a Brownian motion is still useful, due to convenient continuity properties of the solution map with respect to the Volterra signature.  We first introduce the definition of the first level Volterra rough path over Brownian motion $W$, which is a mere elaboration of Definition \ref{def of z1}.
\begin{Def}\label{def of z1 B}
Consider a Brownian motion $W :[0,T]\rightarrow \mathbb{R}^m$ and a function $h$ of the form $h^{\tau}_{ts}(r)=(\tau-r)^{-\gamma}\mathbbm{1}_{[s,t]}(r)$ with $\gamma<1/2$. Then for $(s,t,\tau)\in \Delta_{3}$ we define the increment 
    $\bfz^{1,\tau,i}_{ts}=\int_{s}^{t} \left(\tau-r\right)^{-\gamma}dW^{i}_{r}$ as a Wiener integral of the form 
\begin{equation}\label{z1 expression B}
\bfz_{ts}^{1,\tau,i}:=W^{i}(h^{\tau}_{ts}).
\end{equation}
\end{Def}
Similarly to what we did in Lemma \ref{z1 second moment}, let us find a bound for second moment of $\bfz^{1}$.
\begin{lemma}\label{z1 second moment B}
Consider the Volterra rough path $\bfz^{1}$ as given in \eqref{z1 expression B}, and three parameters $\gamma~\in~(0,1)$, and $\eta,\, \zeta\in [0,1]$ satisfying 
\begin{eqnarray}\label{parameter conditions for W}
0<\gamma<\frac{1}{2}, \qquad \text{and} \quad 0\le \zeta\le \inf\left(\frac{1}{2}-\gamma, \eta\right).
\end{eqnarray}
Then for $(s,t,\tau)\in \Delta_{3}$, we have
\begin{equation}\label{2nd moment of z1 B}
\mathbb{E}[(\bfz^{1,\tau,i}_{ts})^{2}]\lesssim \left|\psi^{1}_{(1/2,\gamma)}\left(\tau,t,s\right)\right|^{2},
\end{equation}
while for $(s,t,\tau^{\prime},\tau)\in\Delta_{4}$, we get
\begin{equation}\label{2nd moment of z1 2 B}
\mathbb{E}[(\bfz^{1,\tau\tau^{\prime},i}_{ts})^{2}]\lesssim\left|\psi^{1,2}_{(1/2,\gamma,\eta,\zeta)}\left(\tau,,\tau^{\prime},t,s\right)\right|^{2},
\end{equation}
where $\psi^{1}$ and $\psi^{1,2}$ are given in Notation \ref{notation of psi}.
\end{lemma}
\begin{proof}
In this proof, we will show that \eqref{2nd moment of z1 B} holds for any $(s,t,\tau)\in \Delta_{3}$. Then relation \eqref{2nd moment of z1 2 B} can be proved in a similar way. Toward to this aim, according to Definition \ref{def of z1 B}, we have 
\begin{eqnarray*}
\mathbb{E}[(\bfz^{1,\tau,i}_{ts})^{2}]=\mathbb{E}\left[\left(\int_{s}^{t} \left(\tau-r\right)^{-\gamma}dW^{i}_{r}\right)^{2}\right].
\end{eqnarray*}
Furthermore, recalling that $W$ is a Brownian motion and resorting to \eqref{isometry for W}, we have  
\begin{equation*}
\mathbb{E}\left[\left(\bfz^{1,\tau,i}_{ts}\right)^{2}\right]=\int_{s}^{t} \left(\tau-r\right)^{-2\gamma}dr.
\end{equation*}
Thanks to some elementary calculations similar to \eqref{bound 1}-\eqref{bound 2} in Section~\ref{Volterra rough path driven by fractional Brownian motion} and recalling definition~\eqref{psi 1} for the function $\psi^{1}_{1/2,\gamma}$, we now obtain
\begin{equation}\label{2nd moment of z1 B 1}
\mathbb{E}[(\bfz^{1,\tau,i}_{ts})^{2}]\lesssim \left[|\tau-t|^{-2\gamma}|t-s|\right]\wedge |t-s|^{1-2\gamma}=\left|\psi^{1}_{\frac{1}{2},\gamma}(\tau,t,s)\right|^{2}.
\end{equation}
This is the desired result \eqref{2nd moment of z1 B}.
\end{proof}
Next we turn our attention to construct the second level Volterra rough path over a Brownian motion. 
\begin{Def}\label{the second level Brownian motion Volterra rough path} 
We consider a Brownian motion $W :[0,T]\rightarrow \mathbb{R}^m$, and the first level of the Volterra rough path $\bfz^{1,\tau}$ defined by \eqref{z1 expression B}. As in Definition~\ref{def of z1 B}, we assume that $\gamma<\frac{1}{2}$. Then for $(s,r,t,\tau)\in \Delta_{4}$, we set
\begin{equation}\label{def of u W}
 u^{\tau,i}_{ts}(r)=(\tau-r)^{-\gamma}\bfz^{1,r,i}_{rs}\mathbbm{1}_{[s,t]}(r).
 \end{equation}
 With this notation in hand, we define the increment $\bfz^{2,\tau}_{ts}$ as an It\^o integral of the form 
\begin{equation}\label{z2 ij W}
 \bfz_{ts}^{2,\tau,i,j}=\int_{s}^{t} u^{\tau,i}_{ts}(r) \, \delta^{\diamond}W^{j}_{r}, \quad \text{for} \quad i,j \in \{1,2, \ldots, m\}.
 \end{equation}
Having the Definition \ref{def of z1 B} of $\bfz^{1,\tau}$ in mind when considering the process $u^{\tau,i}$ in \eqref{def of u W}, we get that $\bfz^{2,\tau}$ in \eqref{z2 ij W} is rewritten as 
\begin{equation*}
 \bfz_{ts}^{2,\tau,i,j}=\int_{s}^{t} \int_{s}^{r}(\tau-r)^{-\gamma}(r-l)^{-\gamma} \,\delta^{\diamond}W^{i}_{l}\,\delta^{\diamond}W^{j}_{r}, \quad \text{for}\quad i,j \in \{1,2, \ldots, m\}.
\end{equation*}
\end{Def}
 
\begin{remark}\label{qv}
Observe that in Definition \ref{the second level Brownian motion Volterra rough path} we have chosen to introduce $\bfz^{2,\tau}_{ts}$ as an It\^o-type integral.  This is in contrast with the fBm case with $H>\frac{1}{2}$, for which \eqref{z2 ii} had to be understood in the Stratonovich sense. As mentioned at the beginning of this section, this is due to the fact that the Stratonovich correction terms for $\bfz^{2}$ are diverging, which again is a consequence of the fact that the covariation between a singular fractional Brownian motion and a Brownian motion is diverging. This has also been noted in \cite{Bayer2020}, where infinite renormalization procedures was proposed to deal with this problem in a regularity structures framework. 
We illustrate the issue  in the following computations. 

For $i=1,\ldots, m$ assume that $\bfz^{2,\tau,i,i}_{ts}$ is defined in the Stratonovich sense, written as 
\begin{eqnarray}
\bfz^{2,\tau,i, \textsc{s}}_{ts}=\int^t_s u^{\tau,i}_{ts}(r) \, dW^{i}_r,
\end{eqnarray}
with $u^{\tau,i}_{ts}$ given in \eqref{def of u W} and $dW^{i}$ denoting the Stratonovich differential. Then standard considerations about It\^o-Stratonovich corrections reveal that 
\begin{equation}\label{relation integral W}
\int^{t}_{s}u^{\tau,i}_{ts}(r) \, dW^{i}_{r}=\int^{t}_{s}u^{\tau,i}_{ts}(r) \, \delta^{\diamond}W^{i}_{r}
+\frac{1}{2}\langle u^{\tau,i}_{ts} , W^{i} \rangle_{ts},
\end{equation}  
where  $\langle u^{\tau,i}_{ts} , W^{i} \rangle_{ts}$ denotes the quadratic variation of  $u^{\tau,i}_{ts}$ and $W^{i}$ over the interval $[s,t]$. 

Let us now analyze the quadratic variation term in \eqref{relation integral W}, which can be defined through a discretization procedure. Namely let $\mathcal{P}$ designate a generic partition of $[s,t]$, and $[\tilde{r},r]$ a typical interval of the partition $\mathcal{P}$. Then a classical way to define the quadratic variation is through the following limit in $L^{2}(\Omega)$:
\begin{equation}\label{qv of uW}
\langle u^{\tau,i}_{ts} , W^{i} \rangle_{ts}=\lim_{|\mathcal{P}|\to 0}\sum_{[\tilde{r},r]\in \mathcal{P}} \left(u^{\tau,i}_{ts}(r)-u^{\tau,i}_{ts}(\tilde{r})\right)W^{i}_{r\tilde{r}}.
\end{equation}
Next we can decompose the right hand side of \eqref{qv of uW} in order to get 
\begin{equation}\label{a1}
\langle u^{\tau,i}_{ts} , W^{i} \rangle_{ts}=\lim_{|\mathcal{P}|\to 0}\sum_{[\tilde{r},r]\in \mathcal{P}}(M_{r\tilde{r}}W^{i}_{r\tilde{r}}+V_{r\tilde{r}}W^{i}_{r\tilde{r}}),
\end{equation}
where $M_{r\tilde{r}}$ and $V_{r\tilde{r}}$ are respectively defined by  
\begin{eqnarray*}
M_{r\tilde{r}}=\int_{\tilde{r}}^{r}(\tau-r)^{-\gamma}(r-l)^{-\gamma}dW^{i}_{l},\quad
V_{r\tilde{r}}=\int_{s}^{\tilde{r}}\left[(\tau-r)^{-\gamma}(r-l)^{-\gamma}-(\tau-\tilde{r})^{-\gamma}(\tilde{r}-l)^{-\gamma}\right]dW^{i}_{l}.
\end{eqnarray*}
Starting from \eqref{a1}, we let the patient reader check that the relation below holds in $L^{2}(\Omega)$:
\begin{eqnarray*}
\lim_{|P|\to 0}\sum_{[\tilde{r},r]\in\mathcal{P}}V_{r\tilde{r}}W^{i}_{r\tilde{r}}=0.
\end{eqnarray*}
However the term $\sum_{[\tilde{r},r]\in\mathcal{P}}M_{r\tilde{r}}W^{i}_{r\tilde{r}}$ in \eqref{a1} is more problematic. Specifically, it can be shown (tedious details are left again to the reader for sake of conciseness) that the following quantity converges in $L^{2}(\Omega)$ as $|\mathcal{P}|\to 0$:
\begin{eqnarray*}
\sum_{[\tilde{r},r]\in \mathcal{P}}\left(M_{r\tilde{r}}W^{i}_{r\tilde{r}}-c_{\gamma}(r-\tilde{r})^{1-\gamma}(\tau-r)^{-\gamma}\right),
\end{eqnarray*}
where $c_{\gamma}=(1-\gamma)^{-1}$. Nevertheless, one can simply check that 
\begin{eqnarray*}
\lim_{|\mathcal{P}|\to 0}\sum_{[\tilde{r},r]\in\mathcal{P}}(r-\tilde{r})^{1-\gamma}(\tau-r)^{-\gamma}=\infty.
\end{eqnarray*}
Hence the quantity $\sum_{[\tilde{r},r]\in \mathcal{P}} M_{r\tilde{r}}W^{i}_{r\tilde{r}}$ is also divergent in $L^{2}(\Omega)$. This proves that the quadratic variation in \eqref{a1} is divergent, and thus going back to \eqref{relation integral W} we get that $\bfz^{2,\tau,i,\textsc{s}}_{ts}$ cannot be defined in the Stratonovich sense. 
\end{remark}
We now adapt the computations of Proposition \ref{second moment of z2} in order to estimate the second moment of the increment $\bfz^{2,\tau,i,j}$.
\begin{prop}\label{second moment of z2 W}
Consider the second level $\bfz_{ts}^{2,\tau}$ of the Volterra rough path, as defined in 
\eqref{z2 ij W}. Recall~ that the parameters $\gamma \in (0,1)$, and $\eta, \zeta\in [0,1]$ satisfy relation \eqref{parameter conditions for W}. Then for $(s,t,\tau)\in~ \Delta_{3}$, we have
\begin{equation}\label{prob reg W}
\mathbb{E}\left[\left(\bfz^{2,\tau}_{ts}\right)^{2}\right]\lesssim \left|\psi^{1}_{(1-\gamma,\gamma)}(\tau,t,s)\right|^{2}.
\end{equation}
For $(s,t,\tau^{\prime},\tau)\in\Delta_{4}$, we get
\begin{equation}\label{prob reg 2 W}
\mathbb{E}\left[\left(\bfz^{2,\tau\tau^{\prime}}_{ts}\right)^{2}\right]\lesssim \left|\psi^{1,2}_{(1-\gamma,\gamma,\eta,\zeta)}(\tau,\tau^{\prime},t,s)\right|^{2},
\end{equation}
For both \eqref{prob reg W} and \eqref{prob reg 2 W}, we recall that $\psi^{1}$ and $\psi^{1,2}$ are given in Notation \ref{notation of psi}. 
\end{prop} 
\begin{proof}
This proof is very similar to the proof of Lemma \ref{z1 second moment B}. We will prove \eqref{prob reg W}, and let the patient reader show that \eqref{prob reg 2 W} holds for $(s,t,\tau^{\prime},\tau)\in\Delta_{4}$.  For $(s,t,\tau)\in \Delta_3$ we have 
\begin{equation*}
\mathbb{E}\left[\left(\bfz^{2,\tau,i,j}_{ts}\right)^{2}\right]=\mathbb{E}\left[\left(\int_{s}^{t} u^{\tau,i}_{ts}(r) dW_{r}^{j}\right)^{2}\right].
\end{equation*} 
Hence according to It\^o's isometry, we obtain
\begin{equation}\label{2nd moment z2 expression W}
\mathbb{E}\left[\left(\bfz^{2,\tau,i,j}_{ts}\right)^{2}\right]=\int_{s}^{t}\mathbb{E}\left[\left(u^{\tau,i}_{ts}\right)^{2}\right]dr.
\end{equation}
Moreover recalling the definition \eqref{def of u W} of $u$, we get 
\begin{equation}\label{2nd moment u expression W}
\mathbb{E}\left[\left(u^{\tau,i}_{ts}\right)^{2}\right]=\mathbb{E}\left[(\tau-r)^{-2\gamma}\left(\bfz^{1,r,i}_{rs}\right)^{2}\mathbbm{1}_{[s,t]}(r)\right]=(\tau-r)^{-2\gamma}\mathbb{E}\left[\left(\bfz^{1,r,i}_{rs}\right)^{2}\right]\mathbbm{1}_{[s,t]}(r).
\end{equation}
Thanks to \eqref{2nd moment of z1 B 1}, we have $
\mathbb{E}\left[\left(\bfz^{1,r,i}_{rs}\right)^{2}\right]\lesssim (r-s)^{1-2\gamma}
$. Then relation \eqref{2nd moment u expression W} reads
\begin{equation}\label{2nd moment u W}
\mathbb{E}\left[\left(u^{\tau,i}_{ts}\right)^{2}\right]\lesssim (\tau-r)^{-2\gamma}(r-s)^{1-2\gamma}\mathbbm{1}_{[s,t]}(r).
\end{equation}
Eventually plugging \eqref{2nd moment u W} into \eqref{2nd moment z2 expression W}, we get 
\begin{equation}\label{2nd moment z2 W}
\mathbb{E}\left[\left(\bfz^{2,\tau}_{ts}\right)^{2}\right]\lesssim\int_{s}^{t} (\tau-r)^{-2\gamma}(r-s)^{1-2\gamma}dr.
\end{equation}
In order to find a bound for the right hand side of \eqref{2nd moment z2 W}, the procedure is very similar to \eqref{318}-\eqref{bound 2} in Proposition \ref{z1 second moment}. We finally obtain
\begin{equation*}
\mathbb{E}\left[\left(\bfz^{2,\tau}_{ts}\right)^{2}\right]\lesssim \left[|\tau-t|^{-2\gamma}|t-s|^{2-2\gamma}\right]\wedge |t-s|^{2-4\gamma}=\left|\psi^{1}_{(1-\gamma,\gamma)}(\tau,t,s)\right|^{2},
\end{equation*}
where we have appealed the definition \eqref{psi 1} of $\psi^{1}$ for the second identity of the above equation. This is the desired result \eqref{prob reg W}.
\end{proof}
With Definition \ref{def of z1 B} and \ref{the second level Brownian motion Volterra rough path} in hand, we have constructed a Volterra rough path family $\{\bfz^{1,\tau},\bfz^{2,\tau}\}$ over a Brownian motion,  and we have also upper bounded their second moment in Lemma \ref{z1 second moment B} and Proposition \ref{second moment of z2 W}. In the following, we close this paper with verifying that $\{\bfz^{1,\tau},\bfz^{2,\tau}\}$ satisfies Definition \ref{hyp 2}. Let us first state that $\bfz^{1,\tau}$ satisfies all properties that mentioned in Definition \ref{hyp 2}.
\begin{prop}\label{z1 space B}
Consider the increment $\bfz^{1,\tau}$ introduced in Definition \ref{def of z1 B}. Then for any $\alpha\in~(0,\frac{1}{2})$, and  $\zeta, \eta \in[0,1]$ satisfying the relation \eqref{condition parameters 2}, we have 
\begin{enumerate}[wide, labelwidth=!, labelindent=0pt, label=\emph{(\roman*)}]

\item
$\bfz^{1,\tau}$ is almost surely in the Volterra space $\mathcal{V}^{(\alpha,\gamma,\eta,\zeta)}(\Delta_{3};\mathbb{R}^{m})$, where $\mathcal{V}^{(\alpha,\gamma,\eta,\zeta)}(\Delta_{3};\mathbb{R}^{m})$ is introduced in Definition \ref{Volterra space}.

\item
For all $p\ge 1$ we have that 
\begin{equation}\label{2p moment z1 B}
\mathbb{E}\left[\|\bfz^{1}\|_{(\alpha,\gamma,\eta,\zeta)}^{2p}\right]<\infty.
\end{equation}

\item
Recalling the definition \eqref{delta} of $\delta$, then $\delta_{m}\bfz^{1,\tau}_{ts}$ satisfies relation \eqref{hyp b}. Namely almost surely we have
\begin{equation}\label{prob chen z1 B}
\delta_{m}\bfz^{1,\tau,i}_{ts}=0,\,\, \text{for all}\,\, (s,m,t,\tau)\in \Delta_{4} .
\end{equation}
\end{enumerate}
\end{prop}
The proof is very similar to the proof as for Proposition \ref{z1 space}-\ref{almost surely delta z1}, we let the patient reader check the details. Similarly, we obtain the following Proposition for the second order integral $\bfz^{2}$.
 
\begin{prop}\label{delta z2 space B}
Consider the second level $\bfz^{2,\tau}$ of the Volterra rough path as defined in 
\eqref{z2 ij W}. Then the following properties hold for any $\alpha\in~(0,\frac{1}{2})$, and  $\zeta, \eta \in[0,1]$ satisfying the relation \eqref{condition parameters 2}.
\begin{enumerate}[wide, labelwidth=!, labelindent=0pt, label=\emph{(\roman*)}]

\item
$\bfz^{2,\tau}$ is almost surely an element of $\mathcal{V}^{(2\rho+\gamma,\gamma,\eta,\zeta)}$ for all $(\eta,\zeta)\in \ca_N$ as defined in \eqref{def of A}.

\item 
For all $p\ge 1$ and $(\eta,\zeta)\in \ca_N$ we have that 
\begin{equation}\label{2p moment z2 B}
\mathbb{E}\left[\|\bfz^{2}\|_{(2\rho+\gamma,\gamma,\eta,\zeta)}^{2p}\right]<\infty.
\end{equation}

\item
Recalling the definition \eqref{delta} of $\delta$, then $\delta_{m}\bfz^{2,\tau}_{ts}$ satisfies relation \eqref{hyp b}, that is 
\begin{equation}\label{prob chen B}
\delta_{m}\bfz^{2,\tau,i,j}_{ts}=\bfz^{1,\tau,j}_{tm}\ast \bfz^{1,\cdot,i}_{ms}, \quad \text{for all}\,\,(s,m,t,\tau)\in \Delta_{4} \quad \text{a.s.}.
\end{equation}
\end{enumerate}

\end{prop}

\begin{proof}
As the proof is very similar to various proofs Section \ref{Sec: prop of volterra RP} we only give a sketch of the method here, and refer to  equivalent proofs in this section for more details.

First it is readily seen that $(i)$ follows from $(ii)$. 
To prove $(ii)$ we will resort to the Volterra GRR lemma \ref{GRR lemma volterra}, in combination with the moment estimates obtained in Proposition \ref{second moment of z2 W}. 
Recall that since $\bfz^2$ is an element of the second chaos of the fBm, the $L^p$ norms are equivalent. Therefore from  the moment estimates obtained in Proposition \ref{second moment of z2 W}, we have that 
\begin{equation}\label{eq:BDG Brownian}
\begin{aligned}
    \EE[(\bfz^{2,\tau}_{ts})^{2p}]&\lesssim \left|\psi^1_{(1-\gamma,\gamma)}(\tau,t,s)\right|^{2p}
    \\
     \EE[(\bfz^{2,\tau\tau'}_{ts})^{2p}]&\lesssim \left|\psi^{1,2}_{(1-\gamma,\gamma,\eta,\zeta)}(\tau,\tau',t,s)\right|^{2p}. 
\end{aligned}
\end{equation}
Invoking relation relation \eqref{eq:bound 1}, we can proceed directly in the same way as in the proof of Proposition~\ref{z2 space} (note that in this case $\rho=\frac{1}{2}$). Combining with the bounds in \eqref{eq:BDG Brownian}, this proves the claim in $(ii)$. 
Claim $(iii)$ can be shown in the same spirit as Proposition \ref{delta z2 relation}, although the integral must now be interpreted in the It\^o sense. Thus Step 1 of the proof of Proposition \ref{delta z2 relation} is exactly the same, while in Step 2 one must consider the integration argument in the It\^o sense. This follows by classical It\^o integration considerations. Step 3 follows by exactly the same arguments. 
This concludes the proof. 
\end{proof}

\subsection{Further extensions and concluding remarks}
We have provided a construction of the Volterra rough path $(\bfz^1,\bfz^2)$ when the driving process is a fractional Brownian motion with $H>\frac{1}{2}$ or a Brownian motion, and the Volterra kernel is allowed to be singular. This corresponds to the  Volterra rough path needed in order to deal with the regularity regime $\alpha-\gamma\geq \frac{1}{3}$, constructed in~\cite{HarangTindel}. It is desirable to extend this construction further to also include higher order components of the signature. As illustrated in the article \cite{HarangTindelWang} and \cite{bruned2021ramification}, in such an extension one will need to deal with several different types of iterated integrals. This abundance of necessary iterated integrals stems from the non-geometeric nature of the Volterra rough path. A more systematic analysis based on related algebraic structures, together with the tools based on Malliavin calculus invoked in the current article, is therefore needed to deal with this problem. 

It is also natural to consider the case of rough fractional Brownian motions as the driving noise (i.e. $H<\frac{1}{2}$). However, the techniques used here, based on the integrability of the mixed partial derivative of the covariance function $R(s,t)$, will no longer work. One will therefore need to use new tools to handle this issue. We  expect that techniques inspired by the results in \cite{Friz2016}, in combination with sewing techniques for Volterra covariance functions developed in \cite{BenthHarang2020}, would prove useful to this aim. However, we leave this problem for  future consideration.

\bigskip

\bibliographystyle{plain}
\bibliography{bib}

\end{document}